\documentclass[12pt]{amsart}

\usepackage{graphicx,amssymb,latexsym,amsfonts,txfonts,amsmath,amsthm}
\usepackage{pstricks,geometry,hyperref,amsmath}

\advance\hoffset by -0.9 truecm

\begin{document} 
\newcommand{\s}{\vspace{0.2cm}} 
\newcommand{\bbb}{\mathbb} 
\newcommand{\Ch}{\widehat{\mathbb C}} 
\newcommand{\Hh}{{\mathbb H}^3} 
\newcommand{\Hw}{{\mathbb H}^2} 
\newcommand{\ngn}{^{-1}} 
\newcommand{\noi}{\noindent} 
\newcommand{\PGLR}{PGL(2,{\mathbb R})} 
\newcommand{\PPGLR}{P^+GL(2,{\mathbb R})} 
\newcommand{\PSLC}{PSL(2,{\mathbb C})} 
\newcommand{\PSLR}{PSL(2,{\mathbb R})} 
\newcommand{\SLR}{SL(2,{\mathbb R})} 
\newcommand{\Stab}{\mathop{\rm Stab}\nolimits} 
\newcommand{\MS}{{\mathcal M}{\mathcal S}_{g}}

\newtheorem{theo}{Theorem} 
\newtheorem{prop}[theo]{Proposition}
\newtheorem{coro}[theo]{Corollary}
\newtheorem{lemm}[theo]{Lemma}
\newtheorem{example}[theo]{Example}
\theoremstyle{remark}
\newtheorem{rema}[theo]{\bf Remark} 
\newtheorem{defi}[theo]{\bf Definition}

\title{Real Structures on Marked Schottky Space} 
\date{\today}

\author{Rub\'en A. Hidalgo} 
\address{Departamento de Matem\'atica y Estad\'{\i}stica, Universidad de La Frontera. Casilla 54-D, Temuco, Chile}
\email{ruben.hidalgo@ufrontera.cl.cl} 

\author{Sebastian Sarmiento}
\address{Departamento de Matem\'atica, Universidad T\'ecnica Federico Santa Mar\'{\i}a, Valpara\'{\i}so, Chile} 
\email{seba.sarmient@gmail.com}

\thanks{Partially supported by Projects Fondecyt 1150003 and Anillo ACT 1415 PIA-CONICYT} 

\begin{abstract}
Schottky groups are exactly those Kleinian groups providing the regular lowest planar uniformizations of closed Riemann surfaces and also the ones providing to the interior of a handlebody  of a complete hyperbolic structure with injectivity radius bounded away from zero.
 The space parametrizing quasiconformal deformations of Schottky groups of a fixed rank $g \geq 1$ is the marked Schottky space $\MS$; this being a complex manifold of dimension $3(g-1)$ for $g \geq 2$ and being isomorphic to the punctured unit disc for $g=1$.
In this paper we provide a complete description of the real structures of $\MS$, up to holomorphic automorphisms, together their real part.
\end{abstract}

\maketitle

%%%%%%%%%%%%%%%%%%%%
%%%%%%%%%%%%%%%%%%%%
\section{Introduction and main results}
A real structure on an $n$-dimensional complex manifold $X$ is provided by an anti-holomorphic automorphism $J:X \to X$ of order two; its fixed points being the corresponding real points. The set of real points is the real part of $J$ and, if not empty,  it is a (possible disconnected) real manifold of dimension $n$. 

The real part of a real structure on $X$ is a natural invariant and the problem of computing it is the so-called Harnack's problem \cite{Krasnov}. Two problems arise in this context: (i) to compute the number of connected components of the real part and (ii) the determination and description of the different real structures that $X$ may have, up to conjugation by holomorphic automorphisms.

If $n=1$ and $X$ is compact, that is, when $X$ is a closed Riemann surface of some genus $g$, then the number of connected components of the real part of a real structure on $X$ is contained in the set $\{0,1\ldots,g+1\}$ \cite{Harnack}. Also, in this case, if $g$ is even, then it has at most four real structures up to holomorphic conjugacy \cite{GI} (in genus zero there are exactly two); for odd genus there are upper bounds in \cite{BGI}. For $n \geq 2$ it seems that there is no known similar results, with the exception of some particular situations; for instance, 
the $n$-dimensional complex projective space ${\mathbb P}^{n}_{\mathbb C}$, where $n \geq 2$ even, has exactly one real structure up to conjugation by holomorphic automorphisms and, for $n \geq 3$ odd, there are two (this can be seen by checking at $(n+1)/2$ general points). 

Another (non-compact) example is given by 
the Teichm\"uller space ${\mathcal T}(S)$ of a closed Riemann surface $S$ of genus $g \geq 1$ (a finite dimensional simply-connected complex manifold parametrizing all Riemann surfaces structures on $S$ up to isotopy), where the description and  classification of its real structures (together their real parts) have been done in several papers; see for instance \cite{Andreotti,B-S,Seppala1,Seppala2,Seppala3,  Weichhold} (in Section \ref{Sec:Teichmuller} we provide a short summary as a matter of completeness). 

Interesting examples of non-compact finite dimensional complex manifolds (in general non simply-connected ones) are given by the quasiconformal deformation space ${\mathcal Q}(K)$ of a finitely generated Kleinian group $K$ with non-empty region of discontinuity (see Section \ref{Sec:Quasiconformal}). In the particular case that $K$ is a torsion-free co-compact Fuchsian group, the manifold ${\mathcal Q}(K)$ is isomorphic to ${\mathcal T}(S) \times {\mathcal T}(S)$, where $S$ is the Riemann surface uniformized by the group; as a consequence, a description of its real structures can be obtained by the knowledge of those of ${\mathcal T}(S)$. But for more general Kleinian groups, the description and classification of its real structures up to holomorphic conjugation seems to be unknown. 

In this paper we assume $K$ to be a Schottky group of rank $g \geq 1$ (see Section \ref{Sec:prelim}). A main interest on Schottky groups is that these are exactly those Kleinian groups providing: (i) the regular lowest planar uniformizations of closed Riemann surfaces and (ii) the complete hyperbolic structure with injectivity radius bounded away from zero to the interior of a handlebody. 

As any two Schottky groups of the same rank are quasiconformally conjugated, their quasiconformal deformation spaces are isomorphic.  A model of ${\mathcal Q}(K)$ is given by the marked Schottky space $\MS$ (see, for instance, \cite{B, Bers, Nag}, and  Section \ref{Sec:Schottkyspace} for details). The space ${\mathcal M}{\mathcal S}_{1}$ is isomorphic to the punctured unit disc $\Delta^{*}=\{0<|z|<1\}$ and, for $g \geq 2$, ${\mathcal M}{\mathcal S}_{g}$ is a non simply-connected complex manifold of dimension $3(g-1)$ (a domain of holomorphy of ${\mathbb C}^{3g-3}$ \cite{Nag}).

As ${\mathcal M}{\mathcal S}_{1}$ is isomorphic to $\Delta^{*}$, one sees that each of its real structures is given by the reflection on a diameter. In particular, (i) every real structure has real points, (ii) the real part has two connected components (each one being an arc), and (ii) they are conjugated (in its group of holomorphic automorphisms; rotations about the origin) to the canonical real structure $J_{1}(z)=\overline{z}$.

The aim of this paper is to provide a description of the real structures, up to conjugation by holomorphic automorphisms, and their real parts, of $\MS$ for $g \geq 2$ (see Theorem \ref{realpoints}). 
We observe that each real point of a real structure of $\MS$ is provided by an extended Schottky group of rank $g$ (see Section \ref{Sec:prelim} for the definition). A structural decomposition of extended Schottky groups, in terms of the Klein-Maskit combination theorems \cite{Maskit:Comb}, was obtained in \cite{H-G:ExtendedSchottky} (see Theorem \ref{maintheo}). This structural decomposition permits to observe that topologically conjugated extended Schottky groups are also quasiconformally conjugated \cite{H-G:ExtendedSchottky}. Another consequence of the structural description is Proposition \ref{formulaMg} (see Section \ref{Mg}) which provides the number $M_{g}$ of topologically non-conjugated extended Schottky groups of a fixed rank $g$. In Section \ref{Sec:realstructure} we will see that each extended Schottky group of rank $g$ induces in a natural way  a real structure on $\MS$ with non-empty real part for which one of its connected components is a real analytic copy of its quasiconformal deformation space. We prove that the converse holds and that real points of real structures on $\MS$ are in correspondence with extended Schottky groups. The main result of this paper is the following.

\s
\noindent
\begin{theo}\label{realpoints}
Let $g \geq 2$ be an integer.
\begin{enumerate}
\item Each real point of the marked Schottky space $\MS$ can be naturally identified with an extended Schottky group of rank $g$. In fact, each connected component of the real part of a real structure is real analytic isomorphic to the quasiconformal deformation of the corresponding extended Schottky group representing a real point on such a component.

\item Any two real structures having fixed points corresponding to topologically equivalent extended Schottky groups are conjugate in the group of holomorphic automorphisms of $\MS$.

\item There are exactly $4$ non-conjugate real structures on ${\mathcal M}{\mathcal S}_{2}$; only one of them has empty real part.

\item If $g \geq 3$ and $T_{g}$ is the number of conjugation classes of order two elements of ${\rm Out}(F_{g})$, then there are exactly $T_{g}+1$ non-conjugate real structure on $\MS$. Moreover, each real structure with non-empty real part is induced by an extended Schottky group.
\end{enumerate}
\end{theo}

\s

Part (3) will be proved in Section \ref{obs21} and the first half of Part (4)  in Section \ref{obs22}. The rest of the proof is done in Section \ref{sec:prueba}.

Part (1) of the previous theorem tell us that, for $g \geq 2$, every real structure, with non-empty real part, of $\MS$ is induced by an extended Schottky group of rank $g$. But, there may be topologically non-equivalent extended Schottky groups of the same rank inducing the same real structure (see Section \ref{obs7} and Remark \ref{obs:conjugar}).

We have provide some results in terms of function groups (more than to be restricted to Schottky groups) hoping that  
the used arguments may be easily adapted to provide a similar description for the real structures of other ``nice" Kleinian groups (for instance, finitely generated torsion-free Kleinian groups whose region of discontinuity is connected); see Remark \ref{observacion}. 

\bigskip

This paper is organized as follows. In Section \ref{Sec:Teichmuller}, just a matter of completeness, we recall the classification of real structures on Teichm\"uller space of a closed Riemann surface.
In Section \ref{Sec:prelim} we provide the definitions of extended Kleinian groups, in particular of extended Schottky groups, state the structural description of extended Schottky groups and compute the number of topologically conjugacy classes of extended Schottky groups of a fixed rank. In Section \ref{Sec:Quasiconformal} we recall some of the definitions and of the properties of quasiconformal deformation theory of (extended) Kleinian groups. In Section \ref{Sec:realstructure} we describe some real structures, with non-empty real part, of the quasiconformal deformation space of function groups (which are induced by extended Kleinian groups). In Section \ref{Sec:Schottkyspace} we provide the definition of marked Schottky groups and that of the marked Schottky space $\MS$ (observing that this is a model for the quasiconformal deformation of a Schottky group of rank $g$). Also, it is stated Marden's description of the group of automorphisms of $\MS$ \cite{Marden}, which asserts that, for $g \geq 3$,  the group of its holomorphic automorphisms is naturally isomorphic to the group ${\rm Out}(F_{g})$ of exterior automorphisms of the free group of rank $g$ (for $g=2$ we need to quotient by the ``hyperelliptic involution"). Then we discuss the real structures (introducing first the ``canonical" one), we compute the real structures for ${\mathcal M}{\mathcal S}_{2}$ and we count the number of conjugacy classes of real structures for $\MS$. The final section contains the rest of the proof of Theorem \ref{realpoints}.

%%%%%%%%%%%%%%%%%%%%%
%%%%%%%%%%%%%%%%%%%%%
\section{On real structures on Teichm\"uller space: a review}\label{Sec:Teichmuller}
In this section we recall the description and classification of the real structures of the Teichm\"uller space of a closed Riemann surface as matter of completeness; the reader may skip this and continue to the next section. We first proceed to 
recall the definition of Teichm\"uller space and then we describe the known results concerning its real structures (see, for instance, \cite{Andreotti, B-S, Seppala1, Seppala2, Seppala3, Seppala4, Weichhold} for a more detailed description and classification of real structures).

Let $S$ be a fixed closed Riemann surface of genus $g \geq 2$. A marking of $S$ is a pair $(R,f)$, where $R$ is a closed Riemann surface of genus $g$ and $f:S \to R$ is an orientation-preserving diffeomorphism. Two markings $(R_{1},f_{1})$ and $(R_{2},f_{2})$ are called {\it Teichm\"uller equivalent} if there is a bi-holomorphism $h:R_{1} \to R_{2}$ so that $f_{2}^{-1}hf_{1}$ is isotopic to the identity. The set of equivalence classes of markings of $S$ is called the {\it Teichm\"uller space} ${\mathcal T}(S)$. As  a consequence of Teichm\"uller's theorem (see, for instance, \cite{Nag}), ${\mathcal T}(S)$ is a simply connected complex manifold of dimension $3(g-1)$. 

Let ${\rm Diff}(S)$ be the group of diffeomorphisms of $S$, 
${\rm Diff}^{+}(S)$ be its index two subgroup of orientation-preserving diffeomorphisms and ${\rm Diff}_{0}(S)$ be its normal subgroup of diffeomorphisms isotopic to the identity.
The extended modular group of $S$ is given by the quaotient ${\rm Mod}(S)={\rm Diff}(S)/{\rm Diff}_{0}(S)$ and the modular group is given by 
${\rm Mod}^{+}(S)={\rm Diff}^{+}(S)/{\rm Diff}_{0}(S)$. 

A natural action of ${\rm Mod}(S)$ on ${\mathcal T}(S)$ is given by 
$[\phi]([R,f)]=[R,f\phi^{-1}]$, where $[\phi] \in {\rm Mod}(S)$ and $[R,f] \in {\mathcal T}(S)$. Royden's theorem \cite{Royden} asserts that 
the group ${\rm Mod}(S)$ acts as the full group of holomorphic and anti-holomorphic automorphism of ${\mathcal T}(S)$,
the group ${\rm Mod}^{+}(S)$ provides the group of holomorphic automorphisms, 
and that this action is faithful  only if $g \geq 3$ (faithfulness fails in genus $g=2$ since the  hyperelliptic involution acts trivially). 

Each real structure on ${\mathcal T}(S)$ is induced by an orientation reversing diffeomorphism of order two $\alpha:S \to S$. Note that if $\phi:S \to S$ is a diffeomorphism acting as the identity on ${\mathcal T}(S)$ (for example, for the hyperelliptic involution if $g=2$), then $\alpha$ and $\phi \alpha$ both induce the same real structure. By the Nielsen realization theorem \cite{Kerckhoff}, there is a new Riemann surface structure $R$ on $S$ so that $\alpha$ acts as an anti-holomorphic involution. Now, by lifting this involution to the universal cover space ${\mathbb H}^{2}$, one obtains a non-Eucliden crystallographic (NEC) group of genus $g$ (its  index two orientation-preserving half is the Fuchsian group of genus $g$ uniformizing $R$).  In this way, the real points of real structures on ${\mathcal T}(S)$ can be identified with the NEC groups of genus $g$. In particular, if $g \geq 3$, then the number of real structures on ${\mathcal T}(S)$ is equal to the number of topologically different NEC groups of genus $g$. 

If $g=2$, then, up to isotopy, there are exactly $5$ topologically non-conjugate orientation reversing diffeomorphisms of order two on $S$. But, as previously noted,  if $\tau_{2}=\eta \tau_{1}$, where $\eta$ is the hyperelliptic involution and $\tau_{1}$ is an anti-conformal involution, then both $\tau_{1}$ and $\tau_{2}$ induce the same real structure. Direct computations permit to see that there are exactly four real structures up to conjugation by a holomorphic automorphism on ${\mathcal T}(S)$. 

If $g \geq 3$ and $\phi:S \to S$ is a diffeomorphism acting as the identity on ${\mathcal T}(S)$,
then (as a consequence of Royden's theorems) 
$\phi$ is isotopic to the identity. So, in this case, the number of different conjugacy classes of real structures of ${\mathcal T}(S)$ is equal to the number of topologically different orientation reversing diffeomorphism of order $2$ in $S$; which is equal to, due to the Nielsen realization theorem \cite{Kerckhoff}, the number of topologically different anti-conformal involutions on genus $g$ Riemann surfaces, that is, equal to $[\frac{3g+4}{2}]$. It was shown by Buser-Sepp\"al\"a \cite{B-S} that any two real structures are real analytically conjugated.

We may observe that  the number of non-conjugated real structures of Teichm\"uller space is much bigger than the number of real structures on marked Schottky space. One reason for this phenomena is essentially due to the fact that not every conformal automorphism of a closed Riemann surface lifts to a Schottky uniformization of it.

%%%%%%%%%%%%%%%%
%%%%%%%%%%%%%%%%
\section{Extended Schottky groups}\label{Sec:prelim}
In this section we recall some of the definitions on (extended) Kleinian groups we will need in this paper. In particular, we focus on certain Kleinian and extended Kleinian groups, called Schottky and extended Schottky groups (these last ones will be the real points of marked Schottky space).
A good source on Kleinian groups is, for instance,  the books \cite{Maskit:book,MT}. 

%%%%%%%%%%%%%%%%%%%%%%
\subsection{(Extended) Kleinian groups}
The conformal automorphisms of the Riemann sphere $\widehat{\mathbb C}$ are the M\"obius transformations and its 
anti-conformal ones are the {\it extended M\"obius transformations} (the composition of the standard reflection $j(z)=\overline{z}$ with a M\"obius transformation). Let us denote by $\mathbb M$ the group of M\"obius transformations and by $\widehat{\mathbb M}$ the group generated by ${\mathbb M}$ and $j$. Clearly, $\mathbb M$ is an index two subgroup of $\widehat{\mathbb M}$. By the Poincar\'e extension \cite{Maskit:book}, $\widehat{\mathbb M}$ acts as the group of hyperbolic isometries of hyperbolic space ${\mathbb H}^{3}$; in this case, ${\mathbb M}$ is the group of orientation-preserving ones. 

The M\"obius transformations are classified into {\it parabolic}, {\it loxodromic} (including hyperbolic) and {\it elliptic} transformations. Similarly, the extended M\"obius transformations are classified into {\it pseudo-parabolic} (the square is parabolic), {\it glide-reflection} (the square is hyperbolic), {\it pseudo-elliptic} (the square is elliptic), {\it reflection} (of order two admitting a circle of fixed points on $\widehat{\mathbb C}$) and {\it imaginary reflection} (of order two and having no fixed points on $\widehat{\mathbb C}$) \cite{Maskit:book}.

A  {\it Kleinian group} (respectively, an {\it extended Kleinian group}) is a discrete subgroup of ${\mathbb M}$ (respectively, a discrete subgroup of $\widehat{\mathbb M}$ necessarily containing extended M\"obius transformations); its region of discontinuity  is the open subset of $\widehat{\mathbb C}$ (it might be empty) of points on which it acts discontinuously.  

If $K_{1}<K_{2}<\widehat{\mathbb M}$, where $K_{1}$ has finite index in $K_{2}$, then $K_{1}$ is discrete if and only if $K_{2}$ is discrete (both have the same region of discontinuity). In particular, $K$ is an extended Kleinian group if and only if $K^{+}:=K \cap {\mathbb M}$ is a Kleinian group. A (extended) Kleinian group is {\it non-elementary} if its limit set (the complement of its region of discontinuity) contains at least $3$ points (then infinitely many points). 

If $K$ is a (extended) Kleinian group with non-empty region of discontinuity $\Omega_{K}$, then we may consider the quotient orbifold $\Omega_{K}/K$. If $K$ is Kleinian group acting freely on $\Omega_{K}$ (no non-trivial element of $K$ has a fixed point in $\Omega_{K}$), then $\Omega_{K}/G$ is a Riemann surface. If $K$ is an extended Kleinian group and $K^{+}$ acts freely on $\Omega_{K}$, then $\Omega_{K}/K$ is a Klein surface (with non-empty boundary if and only if $K$ contains reflections). The group $K$ is called {\it geometrically finite} if there is a finite-sided fundamental polyhedron in ${\mathbb H}^{3}$ for $K^{+}$, in particular, it is finitely generated. It is known that $K$ is geometrically finite if and only if the $3$-orbifold $M_{K}=({\mathbb H}^{3} \cup \Omega_{K})/K$ is essentially compact \cite{Maskit:book}.

Two (extended) Kleinian groups, say $K_{1}$ and $K_{2}$, are called {\it topologically equivalent}  if there is a orientation-preserving homeomorphism $W:\widehat{\mathbb C} \to \widehat{\mathbb C}$ so that $WK_{1}W^{-1}=K_{2}$. 

\subsection{(Extended) Function groups}
Let us consider a pair $(K,\Delta)$,  where $K$ is a finitely generated (extended) Kleinian group and 
$\Delta$ is a connected component of its region of discontinuity which is invariant under the action of $K$. If $K=K^{+}$ we say that the above pair is a {\it function group}; otherwise, an {\it extended function group}.
It is clear from the definition that if $(K,\Delta)$ is an extended function group, then $(K^{+},\Delta)$ is a function group; but the reciprocal is not always true.

\subsection{Torsion-free full function groups}
If $K$ is a torsion-free finitely generated Kleinian group with connected region of discontinuity $\Omega_{K}$, then we  call the pair $(K,\Omega_{K})$ a {\it torsion-free full function group}. In this case, the quotient $M_{K}=({\mathbb H}^{3} \cup \Omega_{K})/K$
is a (bordered) Kleinian $3$-manifold, whose interior is the hyperbolic $3$-manifold $M_{K}^{0}={\mathbb H}^{3}/K$ and whose conformal boundary is the (analytically finite) Riemann surface $S_{K}=\Omega_{K}/K$. 

If $M$ is a (bordered) orientable $3$-manifold homeomorphic to $M_{K}$, then we say that $K$ provides a {\it Kleinian structure} on $M$. 

Let $M_{K}$ and $M_{\Gamma}$ be the Kleinian manifolds associated to any two torsion-free full functions groups $(K,\Omega_{K})$ and $(\Gamma,\Omega_{\Gamma})$, respectively. A {\it conformal diffeomorphism} (respectively, {\it anti-conformal diffeomorphism}) between them is an orientation-preserving (respectively, orientation reversing) diffeomorphism $f:M_{K} \to M_{\Gamma}$ whose restriction to the interiors hyperbolic manifolds $f:M^{0}_{K} \to M^{0}_{\Gamma}$ is an isometry (note that necessarily its restriction to the conformal boundaries $f:S_{K} \to S_{\Gamma}$ is a conformal/anticonformal homeomorphism between Riemann surfaces).

\s
%%%%%%%%%%%%%%%%%%%%%
\subsection{Schottky groups}
We declare the trivial group as the Schottky group of rank zero.

A {\it Schottky group} of rank $g \geq 1$ is a Kleinian group with non-empty region of discontinuity, isomorphic to a free group of rank $g$ and whose elements, different from the identity, are loxodromic transformations.  These groups may equivalently be defined in a geometrical way as follows. Let us consider a collection of pairwise disjoint simple loops $C_{1}, C'_{1},\ldots, C_{g}, C'_{g}$ in the Riemann sphere, all of them bounding a common open domain ${\mathcal D}$ of connectivity $2g$. If $A_{1},\ldots,A_{g} \in {\mathbb M}$ are loxodromic elements so that, for each $j=1,\ldots,g$, it holds that (i) $A_{j}({\mathcal D}) \cap {\mathcal D}=\emptyset$ and (ii) $A_{j}(C_{j})=C'_{j}$, then the group $\langle A_{1},\ldots,A_{g}\rangle$ is a Schottky group of rank $g$ (see, for instance, \cite{Chuckrow, Maskit:Schottky groups}). This geometrical description permits to observe that any two Schottky groups of the same rank are topologically equivalent.

If $K$ is a Schottky group of rank $g$, then its region of discontinuity $\Omega_{K}$ is non-empty and connected (so $(K,\Omega_{K})$ is an example of a torsion-free full function group) and the quotient space $S=\Omega_{K}/K$ is a closed Riemann surface of genus $g$. As a consequence of the retrosection theorem \cite{Bers, Koebe}, every closed Riemann surface can be obtained in this way up to isomorphisms. The quotient space $({\mathbb H}^{3} \cup \Omega_{K})/K$ is homeomorphic to a handlebody of genus $g$; we say that $K$ induces a {\it Schottky structure} on the handlebody.

%%%%%%%%%%%%%%%%%%%%
\subsection{Extended Schottky groups and their geometric structure}\label{Sec:extendedSchottkygroups}
An extended Kleinian group whose index two orientation-preserving half is a Schottky group of rank $g$ is called an 
{\it extended Schottky group} of rank $g$. Basic examples of extended Schottky groups are the following ones:

\begin{enumerate}
\item {\it Cyclic extended Schottky groups}: Cyclic groups generated by either a reflection or an imaginary reflection (in which case the index two orientation-preserving half is the trivial group; the Schottky group of genus zero) or by a glide-reflection (the index two orientation-preserving half is a Schottky group of rank one). These are examples of extended Schottky groups of rank zero and one. 

\item {\it Real Schottky groups}:
Extended Kleinian groups $K$ generated by a Schottky group $\Gamma$ keeping invariant a circle $C$ and by the reflection $\tau_{C}$ of such circle (in this case, $K^{+}=\Gamma$). Let $\Delta$ be any of the two open discs bounded by $C$, let $K_{\Delta}$ be the $K$-stabilizer of $\Delta$, $\Omega_{K}$ the region of discontinuity of $K$ and $\Omega_{C}=\Omega_{K} \cap C$. Two cases may happen as described below. 
\begin{enumerate}
\item $K^{+}=K_{\Delta}$ (i.e. $K^{+}=\Gamma$ is a Fuchsian group of the second kind), in which case, $\Omega_{K}/K=(\Delta \cup \Omega_{C})/K_{\Delta}$ is a compact bordered orientable surface of topological genus $h$ and with $m$ boundary components (if $\gamma$ is the rank of $K^{+}$, then $\gamma+1=2h+m$). We say that $K$ is of signature $(+;h;m)$. 

\item $K^{+} \neq K_{\Delta}$. In this case,  $K_{\Delta}$ must contain glide-reflections and $\Omega_{K}/K=(\Delta \cup \Omega_{C})/K_{\Delta}$ is a compact bordered non-orientable surface of topological genus $h$ and with $m$ boundary components (if $\gamma$ is the rank of $K^{+}$, then $\gamma+1=h+m$). We say that $K$ is of signature $(-;h;m)$. 

\end{enumerate}

\end{enumerate}

\s

A geometric structure of extended Schottky groups, in terms of the Klein-Maskit combination theorems \cite{Maskit:Comb}, was provided in \cite{H-G:ExtendedSchottky} using the above basic extended Schottky groups. A different approach was obtained in \cite{Ka-Mc} using $3$-dimensional techniques. 

\s
\noindent
\begin{theo}[Structural decomposition theorem of extended Schottky groups \cite{H-G:ExtendedSchottky}]\label{maintheo}
\mbox{}
\begin{enumerate}
\item Every extended Schottky group is the free product (in the Klein-Maskit's combination theorem sense) of cyclic extended Schottky groups and real Schottky groups.

\item Any group $K$ constructed, in the sense of the Klein-Maskit combination theorems,  using $``a"$ cyclic groups generated by reflections, $``b"$ cyclic groups generated by imaginary reflections,  
$``c"$ cyclic groups generated by loxodromic transformations, $``d"$ cyclic groups generated by glide-reflections and $``e"$ real Schottky groups, is an extended Schottky group if and only if 
$a+b+d+e>0$. Moreover, if the $``e"$ real Schottky groups are of rank $1 \leq \gamma_{1} \leq \cdots \leq \gamma_{e} $, then $K^{+}$ is a Schottky group of rank
$$g=a+b+2c+2d+e-1+\sum_{j=1}^{e} \gamma_{j}.$$
In this case, the tuple $(a,b,c,d,e;\gamma_{1},\ldots,\gamma_{e})$ is called the signature of of $K$.
\end{enumerate}
\end{theo}

\s

The signature of an extended Schottky group is not sufficient to describe the structure of it; one also needs to take care of the topological structure of the real Schottky groups. We next observe that the geometrical factors in Theorem \ref{maintheo} (in particular, the signature of an extended Schottky group) are ``essentially" uniquely determined by the extended Schottky group. 

\s
\noindent
\begin{lemm}\label{unicadescomp}
If $K$ is an extended Schottky group of rank $g$ of signature $(a,b,c,d,e;\gamma_{1},\ldots,\gamma_{e})$, then
(i) the values $a$, $b$, $c+d$, $e$, and (ii) the $e$ real Schottky groups
are uniquely determined by $K$
\end{lemm}
\begin{proof}
The quotient hyperbolic $3$-orbifold (with boundary) ${\mathbb H}^{3}/K$ has exactly $``b"$ points in its interior where it fails to be a manifold, $``a"$ boundary components being geodesic discs, $``e"$ boundary components being geodesic surfaces (these components also determine the structure of each of the $``e"$ real Schottky groups used in the decomposition of $K$). This determines the values of $a$, $b$, $e$ and the structure of the $e$ real Schottky groups. Now, if the $``e"$ real Schottky groups have signatures $(+;h_{1};m_{1}),\ldots, (+;h_{e_{1}};m_{e_{1}})$, $(-;r_{1};n_{1}),\ldots, (-;r_{e_{2}};n_{e_{2}})$, then the quotient compact surface $\Omega_{K}/K$ (where $\Omega_{K}$ is the region of discontinuity of $K$) is  the connected sum of 
$b+2d+r_{1}+\cdots+r_{e_{2}}$ real projective planes with an orientable surface of genus $c+h_{1}+\cdots+h_{e_{1}}$ and it has exactly $a+m_{1}+\cdots+m_{e_{1}}+n_{1}+\cdots n_{e_{2}}$ boundary components. In this way, $c+d$ (the total number of real projective planes in the connected sum) is uniquely determined.
\end{proof}

\s
%%%%%%%%%%%%%%%%%%%
\subsection{The number of topologically equivalence classes of extended Schottky groups}\label{Mg}
The structural decomposition theorem \ref{maintheo} permits to obtain the number of topologically non-conjugated extended Schottky groups of a fixed rank $g$. This part is not needed in the proof of Theorem \ref{realpoints} and  it is included as a matter of general interest.

\s
\noindent
\begin{prop}\label{formulaMg}
If $g \geq 0$ and
$$\Delta_{g}=\left\{(e;\gamma_{1},...,\gamma_{e}): 1 \leq e, \; 1 \leq \gamma_{1} \leq \cdots \leq \gamma_{e}, \; e+\sum_{j=1}^{e}\gamma_{j} \leq g+1 \right\},$$
then the number of topologically different extended Schottky groups of rank $g$ is
$$M_{g}=\left[\frac{g+4}{2}\right]\left[\frac{g+5}{2}\right]-2  \;+ \sum_{f=(e;\gamma_{1},\cdots,\gamma_{e}) \in \Delta_{g}} N_{f} \prod_{j=1}^{t(f)} T_{\gamma_{j}} \left(T_{\gamma_{j}}+1\right)\left(T_{\gamma_{j}}+2\right)\cdots\left(T_{\gamma_{j}}+l_{j}-1\right)/l_{j}!$$
where
$$T_{\gamma}=2 \cdot (2\gamma-1)!/\gamma!, \quad
N_{f}=\left[ \frac{g_{f}+3}{2}\right]
\left[ \frac{g_{f}+5}{2}\right]-\delta_{f},$$
$$g_{f}=g-e-\sum_{j=1}^{e}\gamma_{j},
\quad \delta_{f}=\left\{\begin{array}{ll}
1, & \mbox{ if }  g_{f} \; \mbox{ is odd}\\
0, & \mbox{ if }  g_{f} \; \mbox{ is even}.
\end{array}
\right.
$$

\end{prop}

\s

For instance, $M_{0}=2$ (these correspond to the cyclic groups generated by either a reflection or an imaginary reflection), $M_{1}=6$ and $M_{2}=17$.

\s

\subsubsection{Proof of Proposition \ref{formulaMg}}
Lemma \ref{unicadescomp}, together Theorem \ref{maintheo}, ensures that the topological type of an extended Schottky group $K$ of rank $g$ is uniquely determined by a tuple
$$(*) \quad (a,b,c,d,e;\{(+;h_{1};m_{1}),\ldots, (+;h_{e_{1}};m_{e_{1}}), (-;r_{1};n_{1}),\ldots, (-;r_{e_{2}};n_{e_{2}})\})$$
such that
$$a+b+d+e>0, \quad e=e_{1}+e_{2},$$
$$2h_{1}+m_{1} \leq 2h_{2}+m_{2} \leq \cdots \leq 2h_{e_{1}}+m_{e_{1}}, \quad r_{1}+n_{1} \leq r_{2}+n_{2} \leq \cdots \leq r_{e_{2}}+n_{e_{2}},$$
$$g=a+b+2c+2d+e-1+\sum_{j=1}^{e_{1}} (2h_{j}+m_{j}-1) +\sum_{j=1}^{e_{2}} (r_{j}+n_{j}-1).$$

In this way, to compute $M_{g}$, we should count the number of such tuples.

\s

Let us observe that if $b+d>0$, then it is possible to change the $c$ loxodromic generators by $c$ glide-reflections in the structural decomposition provided in Theorem \ref{maintheo}. It follows that, for $b+d>0$, the above tuple $(*)$ determines the same topological class as the tuple
$$(**) \quad (a,b,0,c+d,e;\{(+;h_{1};m_{1}),\ldots, (+;h_{e_{1}};m_{e_{1}}), (-;r_{1};n_{1}),\ldots, (-;r_{e_{2}};n_{e_{2}})\}).$$

In particular, $M_{g}$ is equal to the cardinality of the set ${\mathcal F}_{g}$ formed of tuples $$(a,b,c,d,e;\{(+;h_{1};m_{1}),\ldots, (+;h_{e_{1}};m_{e_{1}}), (-;r_{1};n_{1}),\ldots, (-;r_{e_{2}};n_{e_{2}})\})$$
such that
$$a+b+d+e>0, \quad e=e_{1}+e_{2},$$
$$2h_{1}+m_{1} \leq 2h_{2}+m_{2} \leq \cdots \leq 2h_{e_{1}}+m_{e_{1}}, \quad 
r_{1}+n_{1} \leq r_{2}+n_{2} \leq \cdots \leq r_{e_{2}}+n_{e_{2}},$$
$$g=a+b+2c+2d+e-1+\sum_{j=1}^{e_{1}} (2h_{j}+m_{j}-1) +\sum_{j=1}^{e_{2}} (h_{j}+m_{j}-1),$$
$$b+d>0 \implies c=0.$$

\s

If ${\mathcal G}_{g}$ is the set of tuples 
$$(a,b,c,d,(e;\gamma_{1},\ldots,\gamma_{e})),$$
where
\begin{enumerate}
\item $a+b+d+e>0$;
\item $1 \leq \gamma_{1} \leq \gamma_{2} \leq \cdots \leq \gamma_{e}$;
\item if $b+d>0$, then $c=0$; and
\item $g=a+b+2(c+d)+e-1+\sum_{j=1}^{e} \gamma_{j}$,
\end{enumerate}
then
$$\Theta:{\mathcal F}_{g} \to {\mathcal G}_{g}$$
which sends the tuple 
$$(a,b,c,d,e;\{(+;h_{1};m_{1}),\ldots, (+;h_{e_{1}};m_{e_{1}}), (-;r_{1};n_{1}),\ldots, (-;r_{e_{2}};n_{e_{2}})\})$$
to the tuple 
$$(a,b,c,d,(e;\gamma_{1},\ldots,\gamma_{e})),$$
where $$\{\gamma_{1},\ldots,\gamma_{e}\}=\{2h_{1}+m_{1}-1,\ldots,2h_{e_{1}}+m_{e_{1}}-1,r_{1}+n_{1}-1,\ldots,r_{e_{2}}+n_{e_{2}}-1\},$$
is a surjective map. This asserts that
$$M_{g}=\sum_{q \in {\mathcal G}_{g}} | \Theta^{-1}(q) |$$

As $\Theta^{-1}((a,b,c,d,(0;-)))$ only consists of the tuple $(a,b,c,d,0;\{-\})$, we may write
$$M_{g}=|{\mathcal G}_{g}^{0}|+\sum_{q=(a,b,c,d,(e;\gamma_{1},\ldots,\gamma_{e})) \in {\mathcal G}_{g}, e\geq 1} | \Theta^{-1}(q) |$$
where ${\mathcal G}_{g}^{0}$ consists of all tuples $(a,b,c,d,(0;-)) \in {\mathcal G}_{g}$.

\s
\noindent
\begin{lemm}
$|{\mathcal G}_{g}^{0}|=\left[\frac{g+4}{2}\right]\left[\frac{g+5}{2}\right]-2$.
\end{lemm}
\begin{proof}
The tuples $(a,b,c,d,0;\{-\}) \in {\mathcal G}_{g}$ are of two types: (i) the ones of the form $(a,0,c,0)$, where $a \geq 1$ and $g+1=a+2c$, and (ii) the ones of the form $(a,b,0,d)$, where $b+d \geq 1$ and $g+1=a+b+2d$.

To count the tuples of type (i), we only need to count the values $a \in \{1,2,\ldots,g+1\}$ so that $g+1-a$ is even. If $g$ is odd, then $a\in \{2,4,\ldots,g+1\}$ and if $g$ is even, then $a \in \{1,3,\ldots,g+1\}$. In this way, we obtain $[(g+2)/2]$ such tuples.

To count the tuples of type (ii) we consider two cases separately:  $b=0$ and $b \geq 1$. 

If $b=0$, then we need to count those tuples $(a,0,0,d)$, where $d \geq 1$ and $g+1-a=2d$. 
If $g$ is odd, then $a\in \{0,2,4,\ldots,g-1\}$, from which we get $(g+1)/2$ tuples. If $g$ is even, then $a \in \{1,3,\ldots,g-1\}$, from which we obtain $g/2$ tuples.

If $b \geq 1$, then (as $a+b \geq 1$) one has that $g+1-(a+b) \in \{0,\ldots,g\}$ must be even. If $g$ is odd, then $a+b \in \{2,4,\ldots,g+1\}$; that is, we have $(g+1)(g+3)/4$ tuples. If $g$ is even, then $a+b \in \{1,3,\ldots,g+1\}$; that is, we have $(g+2)^{2}/4$ tuples.
\end{proof}

\s

Next, we proceed to compute $$\widehat{M}_{g}:=\sum_{q=(a,b,c,d,(e;\gamma_{1},\ldots,\gamma_{e})) \in {\mathcal G}_{g}, e\geq 1} | \Theta^{-1}(q) |$$

Let us observe that all possible tuples $(e;\gamma_{1},...,\gamma_{e})$ that appear inside a tuple in ${\mathcal G}_{g}$ belong to the set $\Delta_{g}$ as previously defined.

\s

Next, to obtain the value of $\widehat{M}_{g}$ we may proceed as follows. For each $f=(e;\gamma_{1},...,\gamma_{e}) \in \Delta_{g}$ let $N_{f}$ be the number of solutions $(a,b,c,d)$ so that $(a,b,c,d;f) \in {\mathcal G}_{g}$ and let $B_{f}$ be the number of $e$-tuples of topologically non-conjugated real Schottky groups of respective ranks $\gamma_{1},\ldots,\gamma_{e}$. Then $\widehat{M}_{g}=\sum_{f \in \Delta_{g}}N_{f}B_{f}$.

\s
\noindent
\begin{lemm}\label{lemita10}
If $f=(e;\gamma_{1},...,\gamma_{e}) \in \Delta_{g}$, then $N_{f}=\left[ \frac{g_{f}+3}{2}\right]
\left[ \frac{g_{f}+5}{2}\right]-\delta_{f}$, where 
$$g_{f}=g-e-\sum_{j=1}^{e}\gamma_{j},
\quad \delta_{f}=\left\{\begin{array}{ll}
1, & \mbox{ if }  g_{f} \; \mbox{ is odd}\\
0, & \mbox{ if }  g_{f} \; \mbox{ is even}
\end{array}
\right.
$$
\end{lemm}
\begin{proof}
A tuple $(a,b,c,d)$ of non-negative integers will satisfy that 
$(a,b,c,d;f) \in {\mathcal G}_{g}$ if  $a+b+2(c+d)-1=g_{f}$. If $b=d=0$, then we are looking at pairs $(a,c)$ of non-negative integers so that $a+2c=1+g_{f}$, from which we observe that there are exactly $1+\left[ \frac{g_{f}+1}{2}\right]$ of them. Now, if $b+d>0$, then (as $c=0$ in this case), we are looking for triples $(a,b,c)$ of non-negative integers so that $a+b+2d=1+g_{f}$. If $d=0$, then $b \in \{1,\ldots, 1+g_{f}\}$ and $a$ is uniquely determined from the value of $b$; so we get $(1+g_{f})$ such triples. If $d>0$, then $d \in \{1,\ldots, \left[ \frac{g_{f}+1}{2}\right]\}$ and, for each such value of $d$ the value of $b \in \{0,1,\ldots, g_{f}+1-2d\}$ and $a$ is uniquely determined; so we obtain 
$\left[\frac{g_{f}+1}{2}\right]\left[ \frac{g_{f}+2}{2}\right]$ such triples.

\end{proof}

\s
\noindent
\begin{lemm}
If $f=(e;\gamma_{1},...,\gamma_{e}) \in \Delta_{g}$,  let $l_{1},...,l_{t(f)}$ be so that
$$\rho_{l_{1}}=\gamma_{1}=\cdots=\gamma_{l_{1}}<\rho_{l_{2}}=\gamma_{l_{1}+1}=\cdots=\gamma_{l_{1}+l_{2}}< \cdots<\rho_{l_{t}}=\gamma_{l_{1}+\cdots+l_{t(f)-1}+1}=\cdots=\gamma_{l_{1}+\cdots+l_{t(f)}}=\gamma_{e}.$$
Then $$B_{f}=\prod_{j=1}^{t(f)} T_{\gamma_{j}} \left(T_{\gamma_{j}}+1\right)\left(T_{\gamma_{j}}+2\right)\cdots\left(T_{\gamma_{j}}+l_{j}-1\right)/l_{j}!$$
\end{lemm}
\begin{proof}

For $n \in \{l_{1},\ldots,l_{t(f)}\}$, we need to count the number of (unordered) $n$-tuples of real Schottky groups of the same rank $\rho_{n}$ up to topological equivalence. We may identify such a set with the set $X_{n,T_{\rho_{n}}}$, where 
$$X_{n,L}=\{(x_{1},...,x_{n}): x_{j} \in \{1,...,L\}\}/{\mathfrak S}_{n}$$ and ${\mathfrak S}_{n}$ is the permutation group on $n$ letters acting naturally by permutation on the coordinates. 

Let $Q(L,n)$ be the cardinality of $X_{n,L}$. 
It can be seen that $Q(L,1)=L$, $Q(L,2)=L(L+1)/2$ and $Q(L,n)=\sum_{j=1}^{L} Q(j,n-1)$. The last equality permits to obtain
$$Q(L,n)=L(L+1)(L+2)\cdots(L+n-1)/n!$$

Now, as consequence of Lemma \ref{lemarealschottky} below, we obtain the desired formula.
\end{proof}

\s
\noindent
\begin{lemm}\label{lemarealschottky}
The number of topologically different real Schottky groups of a fixed rank $\gamma$ is 
$C_{\gamma}=2 \cdot (2\gamma-1)!/\gamma!$
\end{lemm}
\begin{proof}
Take $2\gamma$ pairwise disjoint circles $C_{1},\ldots,C_{2\gamma}$, each one orthogonal to the unit circle and ordered in counterclockwise order, so that $C_{j+1}$ is the rotation of $C_{j}$ by an angle of $2\pi/\gamma$. Next, we need to consider all possible pairings of these $2\gamma$ circles that can be done by some hyperbolic transformation keeping invariant the unit circle (for each pairing we have two possibilities as the pairing map may keep invariant or not the unit disc). Also, as we want those which are topologically different, we must consider as equivalent any two set of pairings obtained one from the other by a rotation of angle an integer multiple of $2\pi/\gamma$. Such number sets of (equivalent) pairing is then 
$$C_{\gamma}=2^{\gamma} (2\gamma-1)(2\gamma-3)\cdots 5 \cdot 3 \cdot 1)/\gamma=2 \cdot (2\gamma-1)!/\gamma!$$
\end{proof}

\s

For instance, the number of topologically non-conjugated real Schottky groups of ranks one, two and three are respectively 
$T_{1}=2$, $T_{2}=6$ and $T_{3}=40$.

\s
%%%%%%%%%%%%%%%%%
%%%%%%%%%%%%%%%%%
\section{Quasiconformal deformation spaces of Kleinian groups}\label{Sec:Quasiconformal}
In this section we recall some definitions and some of the basics on quasiconformal defomation spaces of finitely generated (extended) Kleinian groups. For details one may see e.g. \cite{Ahlfors, A-B, B, Lehto, Lehto-Virtanen, Nag, Royden}.

\subsection{Quasiconformal homeomorphisms}
If $\Omega_{1}, \Omega_{2} \subset \widehat{\mathbb C}$ are two non-empty domains, then an orientation-preserving  homeomorphism
$W:\Omega_{1} \to \Omega_{2}$ is called {\it quasiconformal} if it satisfies the following two properties:
\begin{itemize}
\item[(i)] $W$ has distributional partial derivatives with respect to $z$ and $\overline{z}$ which can be represented by locally integrable functions $W_{z}$ and $W_{\overline{z}}$, respectively, on $\Omega_{1}$;

\item[(ii)] there is a measurable function $\mu:\Omega_{1} \to {\mathbb C}$ (called a {\it complex dilation} of $W$) with $\|\mu\|_{\infty}<1$ ($\| \mbox{ } \|_{\infty}$ being the essential supreme norm) such that $W$ satisfies the Beltrami equation
$$W_{\overline{z}} (z)= \mu(z) W_{z}(z) \quad a.e. \; z \in \Omega_{1}.$$
\end{itemize}

The existence and uniqueness of quasiconformal homeomorphisms is due to Morrey \cite{Morrey}.

\s

\begin{theo}[Measurable Riemann mapping theorem \cite{A-B,Morrey}]\label{A-B}
If $\mu:{\mathbb C} \to {\mathbb C}$ is a measurable function with $\|\mu\|_{\infty}<1$, then there is, and it is  unique,
a quasiconformal homeomorphism $W_{\mu}:\widehat{\mathbb C} \to \widehat{\mathbb C}$, with complex dilation $\mu$, fixing the points $\infty$, $0$ and $1$. Moreover, if $\mu$ vary continuously or real-analytically or holomorphically and $z_0$ is a fixed point, then $W_{\mu}(z_0)$ varies also in the same way.
\end{theo}

\s

Two (extended) Kleinian groups, say $K_{1}$ and $K_{2}$, are called {\it quasiconformally equivalent} if there is a quasiconformal homeomorphism $W:\widehat{\mathbb C} \to \widehat{\mathbb C}$ so that $WK_{1}W^{-1}=K_{2}$. We should observe that, for finitely generated Kleinian groups, to be topologically equivalent is the same as to be quasiconformally equivalent.

\s
%%%%%%%%%%%%%%%%%%%%%%%%
\subsection{Quasiconformal deformation spaces of (extended) function groups}
In this section, $(K,\Delta)$ will be an (extended) function group and let $(K^{+},\Delta)$ be its corresponding function group (note that we may have that $K \neq K^{+}$ or $K = K^{+}$). 

%%%%%%%%%%%%%%%%%
\subsubsection{Beltrami coefficients of $(K,\Delta)$}
We denote by $L^{\infty}(K,\Delta)$ the Banach space  (with the essential supreme norm $\| \mbox{ } \|_{\infty}$) whose elements are those measurable functions $\mu:\widehat{\mathbb C} \to {\mathbb C}$ so that 
$$\left\{ \begin{array}{ll}
\mu(z)=0, & \forall z \in \widehat{\mathbb C}-\Delta\\
\\
\mu(k(z))\overline{k_{z}(z)}=
k_{z}(z) \mu(z), & \mbox{if $k \in K^{+}$}; \\
&\\
\mu(k(z))\overline{k_{\overline{z}}(z)}=k_{\overline{z}}(z) \overline{\mu}(z), & \mbox{if $k \in K-K^{+}$}.\end{array}
\right.
$$

The elements of the unit ball $L^{\infty}_{1}(K,\Delta)$  in $L^{\infty}(K,\Delta)$ are called the {\it Beltrami coefficients} of $(K,\Delta)$. 

By the measurable Riemann mapping theorem, for each Beltrami coefficient $\mu \in L^{\infty}_{1}(K,\Delta)$ there is a unique quasiconformal homeomorphism $W_{\mu}:\widehat{\mathbb C} \to \widehat{\mathbb C}$, with complex dilation $\mu$, fixing the points $0$, $1$ and $\infty$. Now, for each $k \in K^{+}$ (respectively, $k \in K-K^{+}$), the element $k_{\mu}=W_{\mu} k {W_{\mu}}^{-1}$ is again a M\"obius transformation (respectively, an extended M\"obius transformation). If we set $K_{\mu}=W_{\mu} K {W_{\mu}}^{-1}$, then $(K_{\mu},\Delta_{\mu})$ is a (extended) function group, where $\Delta_{\mu}:=W_{\mu}(\Delta)$. The map $\chi_{\mu}:K \to K_{\mu}: k \mapsto k_{\mu}$
is an isomorphism.

%%%%%%%%%%%%%%%%%
\subsubsection{The quasiconformal deformation space ${\mathcal Q}(K,\Delta)$}
Two elements $\mu_{1}, \mu_{2} \in  L^{\infty}_{1}(K,\Delta)$ are  called {\it quasiconformal equivalent} ($\mu_{1} \sim \mu_{2}$), if $\chi_{\mu_{1}}$ and  $\chi_{\mu_{2}}$ are the same isomorphism. If the group $K$ is non-elementary, then this is equivalent to say that $W_{\mu_{1}}$ and $W_{\mu_{2}}$ coincide on the limit set of $K$. 

The space ${\mathcal Q}(K,\Delta)=L^{\infty}_{1}(K,\Delta)/\sim$, of quasiconformal equivalent classes of Beltrami coefficients for $(K,\Delta)$, is called the {\it quasiconformal deformation space} of $K$ supported on $\Delta$.  In the case that $\Delta=\Omega(K)$, we use the notations $L_{1}^{\infty}(K)$ and ${\mathcal Q}(K)$ for the previous spaces. 

As the (extended) function group $K$ is finitely generated, it is well known that ${\mathcal Q}(K,\Delta)$ is (see \cite{Maskit:selfmaps})
\begin{itemize}
\item[(i)] a complex manifold of finite dimension, if $K=K^{+}$, and 
\item[(ii)] a real analytic manifold of finite dimension, if $K \neq K^{+}$.
\end{itemize}

It is also known (see \cite{B, Nag}) that, if $(K,\Delta)$ is a function group with $\Delta$ simply connected, then  ${\mathcal Q}(K,\Delta)$ is a model of the Teichm\"uller space of the orbifold $\Delta/K$. 

\s
\noindent
\begin{rema}
In Section \ref{Sec:Schottkyspace} we review an explicit model of ${\mathcal Q}(K)$, when $K$ is a Schottky group of rank $g$; this being the marked Schottky space ${\mathcal M}{\mathcal S}_{g}$.
\end{rema}

\s
%%%%%%%%%%%%%%%%%%
\subsection{The modular group of $(K^{+},\Delta)$}
In this section, we assume that $(K^{+},\Delta)$ is a function group.
Let $T:\Delta \to S=\Delta/K^{+}$ be a regular planar branched covering with $K^{+}$ as its deck group of cover transformations (if $K^{+}$ acts freely on $\Delta$, then $T$ will be a regular covering map). Since $K^{+}$ is finitely generated, Ahlfor's finiteness theorem \cite{Ahlfors:1964} asserts that $S$ is an analytically finite Riemann orbifold, that is, the complement of a finite number of points of a closed Riemann orbifold with a finite number of cone points.  As $K^{+}$ is non-elementary, the universal cover space of $S$ is the hyperbolic plane ${\mathbb H}^{2}$. Let us consider a universal cover $P:{\mathbb H}^{2} \to S$, whose deck group is given by a Fuchsian group $\Gamma$ of the first kind (that is, all the boundary of ${\mathbb H}^{2}$ is the limit set of $\Gamma$). There is a torsion-free normal subgroup $N \lhd \Gamma$ and a regular planar covering $Q:{\mathbb H}^{2} \to \Delta={\mathbb H}^{2}/N$ with $N$ as deck group so that $P=T Q$.

The regular planar branched covering $T:\Delta \to S$ is (topologically) determined by a finite collection 
of pairwise disjoint loops $\{w_{m}=\alpha_{m}^{k_{m}}\}$, where each $\alpha_{m}$ is a simple loop and $k_{m} \in \{1,2,3,...\}$, in the sense that the above is a highest regular planar branched covering of the orbifold $S$ with the property that all the above loops $w_{m}$ lift to loops in $\Delta$ \cite{Maskit:book}. Such a collection of loops is called a {\it defining set of loops} for $T:\Delta \to S$. Defining set of loops are not unique in general, but in here we have  fixed one of them.

An {\it orbifold diffeomorphism} of $S$ is a diffeomorphism of $S$ which keeps invariant both the set of cone points and preserves the cone orders. These are exactly those difeomorphisms of $S$ which lift under the universal (branched) covering $P:{\mathbb H}^{2} \to S$. Note that an orbifold diffeomorphism which is isotopic to the identity acts as the identity on the cone points and it preserves the orientation. Let us denote by ${\rm Diff}_{orb}(S)$ the group of orbifold diffeomorphisms of $S$ and by ${\rm Diff}_{orb}^{0}(S)$ its (necessarily normal) subgroup of consisting of the orbifold diffeomorphisms isotopic to the identity. The index two orientation-preserving half of ${\rm Diff}_{orb}(S)$ is denoted by ${\rm Diff}_{orb}^{+}(S)$ and consists of those orbifold diffeomorphisms of $S$ that preserve the orientation. The {\it extended modular group} is defined by $Mod_{orb}(S)={\rm Diff}_{orb}(S)/{\rm Diff}_{orb}^{0}(S)$ and its index two orientation-preserving half is the {\it modular group} $Mod_{orb}^{+}(S)={\rm Diff}_{orb}^{+}(S)/{\rm Diff}_{orb}^{0}(S)$. 

Let us denote by $H^{*}(K^{+},\Delta)<Mod_{orb}(S)$ the  subgroup consisting of those classes of orbifold diffeomorphisms that lift under $T$ to diffeomorphisms of $\Delta$. We  denote by $H^{*}(K^{+},\Delta)^{+}$ its index two orientation-preserving half, that is, consisting of those classes of orientation-preserving orbifold diffeomorphisms in $H^{*}(K^{+},\Delta)$. 

By the lifting property, each $h \in H^{*}(K^{+},\Delta)$ induces a natural isomorphism $h_{*}:\Gamma \to \Gamma$ so that $h_{*}(N)=N$, in particular, it defines an isomorphism of $K^{+}=\Gamma/N$. We set by $H(K^{+},\Delta) \lhd H^{*}(K^{+},\Delta)$ the normal subgroup formed by the classes of those orbifold diffeomorphisms which induce the identity isomorphism of $K^{+}$; we set by $H(K^{+},\Delta)^{+}$ its index two orientation-preserving half.

If we denote by $Mod_{H(K^{+},\Delta)^{+}}(S)$ the normalizer in $Mod_{orb}(S)$ of $H(K^{+},\Delta)^{+}$, then clearly
$$H(K^{+},\Delta)^{+} \lhd H^{*}(K^{+},\Delta)< Mod_{H(K^{+},\Delta)^{+}}(S).$$

\s

\begin{prop}\label{prop1}
$H^{*}(K^{+},\Delta)= Mod_{H(K^{+},\Delta)^{+}}(S)$.
\end{prop}
\begin{proof}
Let $[h] \in Mod_{H(K^{+},\Delta)^{+}}(S)$ and 
consider a defining set of loops for $T:\Delta \to S$, say  $\{w_{m}=\alpha_{m}^{k_{m}}\}$, where $\alpha_{m}$ is a simple loop and $k_{m} \in \{1,2,3,...\}$. If $\alpha_{m}$ is a small loop around a disc containing one of the cone points (or one of the punctures), then $h(\alpha_{m})$ will be a  small loop around a disc containing another cone point (or a puncture), with the same cone order. It follows that $h(\alpha_{m})^{k_{m}}$ lifts to a loop under $T:\Delta \to S$. If $\alpha_{m}$ does not bound a disc (or punctured disc) on $S$, then we consider its Dehn-twist $D_{\alpha_{m}}:S \to S$. It is not difficult to see that $D^{k_{m}}_{\alpha_{m}} \in H(K^{+},\Delta)^{+}$. As $h D^{k_{m}}_{\alpha_{m}} h^{-1} = D^{k_{m}}_{h(\alpha_{m})} \in H(K^{+},\Delta)^{+}$, we see that necessarily $h(\alpha_{m})^{k_{m}}$ lifts to a loop under $T:\Delta \to S$. It follows that $h$ preserves the collection of loops that lift to loops under $T$, in particular,  $[h] \in H^{*}(K^{+},\Delta)$.
\end{proof}

\s
\noindent
\begin{rema}
As already noted, the quasiconformal deformation space ${\mathcal Q}(\Gamma,{\mathbb H}^{2})$ is a model of the {\it Teichm\"uller space} ${\mathcal T}(S)$ and it is a universal cover space of ${\mathcal Q}(K^{+},\Delta)$ \cite{B}. Maskit \cite{Maskit:selfmaps} noted that the universal covering group is given by $H(K^{+},\Delta)^{+}$, that is, $\pi_{1}({\mathcal Q}(K^{+},\Delta)) \cong H(K^{+},\Delta)^{+}$.
\end{rema}

\s
\noindent
\begin{prop}
If $S=\Delta/K^{+}$ is a closed Riemann surface of genus $g \geq 3$, then the quotient $H^{*}(K^{+},\Delta)/H(K^{+},\Delta)^{+}$ acts on ${\mathcal Q}(K^{+},\Delta)$ as the full group of holomorphic and anti-holomorphic automorphisms and that $H^{*}(K^{+},\Delta)^{+}/H(K^{+},\Delta)^{+}$ acts as the full group of holomorphic automorphisms.
\end{prop}
\begin{proof}
It is well known that the modular group $Mod_{orb}(S)$ acts as group of holomorphic and anti-holomorphic automorphisms of ${\mathcal T}(S)$; the group $Mod_{orb}^{+}(S)$ acting as group of holomorphic automorphisms. Royden's theorem \cite{Royden} asserts that, for $g \geq 3$, $Mod_{orb}(S)$ is the full group of holomorphic and anti-holomorphic automorphisms of ${\mathcal T}(S)$. The desired result now follows from Proposition \ref{prop1}. 
\end{proof}

\s
 
If $\Delta$ is the entire region of discontinuity of $K^{+}$, then we will use the notations $H^{*}(K^{+})$, $H(K^{+})$ and so on.

\s

\begin{rema}
In the particukar case that $K^{+}$ is a Schottky group,  Luft \cite{Luft} proved that $H(K^{+})^{+}$ is generated by Dehn-twists along simple loops on $S=\Delta/K^{+}$ which lift to simple loops to $\Delta$. 
\end{rema}

\s

%%%%%%%%%%%%%%%%%%%%%%%%%%%%
\subsection{The case of torsion-free full function groups}
Let us fix a torsion-free full function group $(K,\Omega_{K})$ and its associated  Kleinian $3$-manifold $M_{K}$. 

A marking of $M_{K}$ is a pair $(M_{\Gamma},f)$, where $M_{\Gamma}$ is the Kleinian $3$-manifold associated to a torsion-free full function group $(\Gamma,\Omega_{\Gamma})$ and $f:M_{K} \to M_{\Gamma}$ is an orientation-preserving diffeomorphism. 

Two markings $(M_{1},f_{1})$ and $(M_{2},f_{2})$ are said to be {\it Teichm\"uller equivalent} if there is a conformal diffeomorphism $h:M_{1} \to M_{2}$ so that $f_{2}^{-1}hf_{1}$ is isotopic to the identity. 

The {\it Teichm\"uller space} of $M_{K}$, denoted as ${\mathcal T}(M_{K})$, is defined as the set of equivalence classes of markings of $M_{K}$. 

Let ${\rm Diff}(M_{K})$ be the group of diffeomorphisms of $M_{K}$, ${\rm Diff}^{+}(M_{K})$ be its index two subgroup of orientation-preserving diffeomorphisms and ${\rm Diff}_{0}(M_{K})$ be its normal subgroup of diffeomorphisms isotopic to the identity. The quotient group ${\rm Mod}(M_{K})={\rm Diff}(M_{K})/{\rm Diff}_{0}(M_{K})$
is called the {\it extended modular group} of $M_{K}$ and ${\rm Mod}^{+}(M_{K})={\rm Diff}^{+}(M_{K})/{\rm Diff}_{0}(M_{K})$ its {\it modular group}. 
An element $[h] \in {\rm Mod}(M_{K})$ acts on ${\mathcal T}(M_{K})$ by the following rule: $[h]([M,f])=[M,fh^{-1}]$. The {\it moduli space} of $M_{K}$ is defined by ${\mathcal M}(M_{K})={\mathcal T}(M_{K})/{\rm Mod}^{+}(M_{K})$.  

The following result states that if $K$ is also geometrically finite, then ${\mathcal Q}(K)$ provides a complex manifold structure for ${\mathcal T}(M_{K})$.

\s
\noindent
\begin{prop}\label{identifica}
If $(K,\Omega_{K})$ is a torsion-free full function group and $K$ is also geometrically finite, then ${\mathcal T}(M_{K})$ can be naturally identified with ${\mathcal Q}(K)$ so that  ${\rm Mod}(M_{K})$ corresponds to the group $H^{*}(K)/H(K)^{+}$, the full group of holomorphic and anti-holomorphic automorphisms of ${\mathcal Q}(K)$.
\end{prop}
\begin{proof}
Let $\pi_{K}:{\mathbb H}^{3} \cup \Omega_{K} \to M_{K}$ be a universal covering with $K$ as deck group. Let $(M_{\Gamma},f)$ be a marking of $M_{K}$ and let 
$\pi_{\Gamma}:{\mathbb H}^{3} \cup \Omega_{\Gamma} \to M_{\Gamma}$ be a universal covering with $\Gamma$ as deck group. We may assume (up to isotopy) that $f:S_{K} \to S_{\Gamma}$ is a quasiconformal diffeomorphism. We may lift the diffeomorphism $f$ to a diffeomorphism $\widehat{f}:{\mathbb H}^{3} \cup \Omega_{K} \to {\mathbb H}^{3} \cup \Omega_{\Gamma}$ satisfying that, for every $k \in K$, it holds that $\widehat{f} \; k = \theta(k) \widehat{f}$, where $\theta:K \to \Gamma$ is an isomorphism of groups. The restriction 
$\widehat{f}: \Omega_{K} \to  \Omega_{\Gamma}$ is a quasiconformal diffeomorphism. As $K$ is geometrically finite, it follows from Marden's isomorphism theorem \cite{Marden} that we 
may assume $\widehat{f}: \Omega_{K} \to  \Omega_{\Gamma}$ to be the restriction of 
a quasiconformal homeomorphism of the Riemann sphere that conjugates $K$ into $\Gamma$.

Conversely, again as a consequence of Marden's isomorphism theorem \cite{Marden}, each quasiconformal diffeomorphism $h:\widehat{\mathbb C} \to \widehat{\mathbb C}$ so that $h K h^{-1}=\Gamma$, extends to an orientation-preserving diffeomorphism $\widehat{h}:{\mathbb H}^{3} \cup \Omega_{K} \to {\mathbb H}^{3} \cup \Omega_{\Gamma}$ keeping the conjugacy property. It follows that $\widehat{h}$ induces a marking of $M_{K}$ making the above two processes inverse to each other.
\end{proof}

\s
%%%%%%%%%%%%%%%%%%%%%%%
\section{Some real structures of quasiconformal deformation spaces}\label{Sec:realstructure}
In this section we describe real structures of quasiconformal deformation spaces defined from extended function groups. 

Let us fix some extended function group $(K,\Delta)$ and let us consider the quasiconformal deformation spaces ${\mathcal Q}(K^{+},\Delta)$ (which is a complex manifold of some complex dimension $d$) and the quasiconformal defomation space ${\mathcal Q}(K,\Delta)$ (which is a real analytic manifold of real dimension $d$). As 
\begin{itemize}
\item[(i)] each Beltrami coefficient for $K$ is also a Beltrami differential for $K^{+}$ and 
\item[(ii)] both $K$ and $K^{+}$ have the same limit set, 
\end{itemize}
there is a natural real analytic embedding 
$$i_{K,\Delta}:{\mathcal Q}(K,\Delta) \to {\mathcal Q}(K^{+},\Delta)$$
so that $[0] \in i_{K,\Delta}({\mathcal Q}(K,\Delta)) \subset {\mathcal Q}(K^{+},\Delta)$.

\s
\noindent
\begin{rema}
More generally, if $(\Gamma,\Delta_{\Gamma})$ is a function group so that $\Gamma_{\mu}=K^{+}$  and $W_{\mu}(\Delta_{\Gamma})=\Delta$, for some 
$[\mu] \in {\mathcal Q}(\Gamma,\Delta_{\Gamma})$, then 
 the above provides a real analytic embedding 
$j:{\mathcal Q}(K,\Delta) \to {\mathcal Q}(\Gamma,\Delta_{\Gamma})$ so that $j([0])=[\mu]$.
\end{rema}

\s

Next we proceed to observe that the real analytic embedded manifold $i_{K,\Delta}({\mathcal Q}(K,\Delta))$ is a connected component of a real structure of ${\mathcal Q}(K^{+},\Delta)$ induced by the group $K$.

\s
\noindent
\begin{theo} \label{teoreal}
Let $(K,\Delta)$ be an extended function group. Then

\begin{enumerate}
\item[(i)] Every extended M\"obius transformation $\tau \in K-K^{+}$ induces a real structure $\tau^{*}$ on  ${\mathcal Q}(K^{+},\Delta)$ defined by 
\begin{equation}\label{eq1}
\tau^{*}([\mu(z)])=\left[ \dfrac{\tau_{\overline{z}}(z)\overline{\mu}(\tau(z))}{\overline{\tau_{\overline{z}}(z)}} \right].
\end{equation}

\item[(ii)] By construction, $i_{(K,\Delta)}({\mathcal Q}(K,\Delta)) \subset {\mathcal Q}(K^{+},\Delta)$ is a connected component of the part of fixed points of the real structure $\tau^{*}$.

\item[(iii)] Moreover, the real structures induced by any two elements of $K-K^{+}$ are the same. 

\end{enumerate}
\end{theo}
\begin{proof}
Parts (i) and (ii) are just a direct computation. Let us proceed to check (iii). Let $\eta, \tau \in K-K^{+}$ and $[\mu] \in {\mathcal Q}(K^{+},\Delta)$. If $T \in K^{+}$ is such that
$\eta=T\tau$, then the equality $\eta^{*}=\tau^{*}$ (as defined above) is consequence of $\mu(T(z))\overline{T_{z}(z)}=\mu(z) T_{z}(z)$ and $(T\tau)_{\overline{z}}=T_{z}(\tau) \tau_{\overline{z}}$.
\end{proof}

\s
\noindent
\begin{rema}
In general, one should not expect that every real structure, with non-empty real part, of ${\mathcal Q}(K^{+},\Delta)$ comes as above (it might be that $\Delta$ has real a structure normalizing $K^{+}$ which is not the restriction of an extended M\"obius transformation). For the case of Schottky groups, part (1) of Theorem \ref{realpoints} asserts that every real structure of the quasiconformal deformation space of a Schottky group, with non-empty real part, is induced from an extended Schottky group as explained above.
\end{rema}

\s

\subsection{Some words on the Schottky situation}
Theorem \ref{realpoints} asserts that every real structure with real points on the quasiconformal deformation space of a Schottky group of rank $g$ is induced by an extended Schottky group of rank $g$ (as indicated in Theorem \ref{teoreal}). This means that the number of such real structures, up to holomorphic automorphisms, is bounded above by the number $M_{g}$ (computed in Theorem \ref{formulaMg}) of topologically non-conjugated extended Schottky groups of rank $g$. On the other hand,  for $g=1$, there is exactly one real structure up to conjugacy by holomorphic automorphisms, but $M_{1}=6$. Similarly, by Theorem \ref{realpoints}, for $g=2$ there are exactly three real structures with real points,  up to conjugacy by holomorphic automorphisms, but $M_{2}=17$.
This discrepancy between the number $M_{g}$ and the number of real structures with real points, up to conjugacy, happens because there are topologically non-conjugated extended Schottky groups inducing the same real structure. Let us make this more precise in the cases $g=1,2$.

\begin{enumerate}

\item Let $K$ be an extended Schottky group of rank one, $K^{+}=\langle A \rangle$ be its index two orientation preserving Schottky group and let $E$ be the unique elliptic transformation of order two commuting with $A$ (both have the same fixed points). If $\eta \in K-K^{+}$, then $K_{\eta}=\langle A, E\eta\rangle$ is an extended Schottky group of rank one, with orientation-preserving half $K^{+}_{\eta}=K^{+}$ (which might or not be topologically conjugated to $K$) and inducing the same real structure on ${\mathcal M}{\mathcal S}_{1}=\Delta^{*}$ as the one induced by $K$. For instance, if we start with the extended Schottky group $K=\langle \tau_{1}(z)=1/\overline{z}, \tau_{2}(z)=4/\overline{z}\rangle$, then $A(z)=4z$ and $E(z)=-z$. In this example, by taking $\eta=\tau_{1}$ ($E\tau_{1}(z)=-1/\overline{z}$ is an imaginary reflection), the above procedure provides the extended Schottky group of rank one $K_{\tau_{1}}=\langle A,E\tau_{1} \rangle=\langle E\tau_{1},\tau_{2} \rangle$. As $K$ uniformizes an bordered annulus and $K_{\tau_{1}}$ uniformizes a bordered M\"obius band, these groups two are not topologically conjugated; but they induces the same real structure. As we will see later, all extended Schottky groups of rank one induces the same real structure.

\item If $K^{+}=\langle A_{1},A_{2}\rangle$ is a Schottky group of rank two, then $E=A_{1}A_{2}-A_{2}A_{1}$ is an elliptic transformation of order two satisfying that $EA_{j}E=A_{j}^{-1}$ (so it induces the identity automorphism of ${\mathcal M}{\mathcal S}_{2}$). If $\Omega$ is the region of discontinuity of $K^{+}$ and $S=\Omega/K^{+}$ is the uniformized Riemann surface of genus two, then $E$ induces the hyperelliptic involution on $S$ \cite{Keen}.  Now, if $K$ is an extended Schottky group whose orientation-preserving half is $K^{+}$ and $\eta \in K-K^{+}$, then $K_{\eta}=\langle K^{+},\eta E \rangle$ still an extended Schottky group of rank two with $K^{+}$ as its orientation-preserving half (which might or not be topologically conjugated to $K$) and both inducing the same real structure of ${\mathcal M}{\mathcal S}_{2}$. Now, we may define that two extended Schottky groups of rank two are twisted-equivalent if we may apply the above procedure to one of them (with some of its extended M\"obius transformations) to obtain an extended Schottky group topologically equivalent to the second one. It happens that there are exactly six twisted equivalence classes. We will see that these define exactly three real structures, up to conjugacy,
\end{enumerate}

\s
%%%%%%%%%%%%%%%%%%%%%%
\section{Marked Schottky spaces and their real structures}\label{Sec:Schottkyspace}
The quasiconformal deformation space of a Schottky group of rank $g \geq 1$ turns out to be a connected complex manifold, of dimension one if $g=1$ and of dimension $3(g-1)$  if $g \geq 2$ \cite{B, Bers, Nag}. As any two Schottky groups of the same rank are quasiconformally equivalent, their respective quasiconformal deformation spaces are isomorphic. Next, we proceed to recall a basic model of this quasiconformal deformation space, called the marked Schottky space. This model is the punctured unit disc for $g=1$ and an open domain of $({\mathbb C}-\{0,1\})^{3g-3}$ for $g \geq 2$ .

\s
%%%%%%%%%%%%%%%%%
\subsection{Marked Schottky groups}
A {\it marked Schottky group} of rank $g \geq 1$ is a tuple $(A_{1},...,A_{g})$, where each $A_{j}$ is a loxodromic transformations and $\langle A_{1},..., A_{g}\rangle$ (the group generated by all these transformations) is a Schottky group of rank $g$. Two marked Schottky groups  $(A_{1},...,A_{g})$ and $(B_{1},...,B_{g})$ are said to be {\it equivalent} if there is a M\"obius transformation $C$ so that $B_{j}=C A_{j} C^{-1}$, for every $j=1,...,g$. The space of equivalence classes of marked Schottky groups of rank $g$, denoted as $\MS$, is called the {\it marked Schottky space of rank $g$}. 

\s
%%%%%%%%%%%%%%%%%
\subsection{Marked Schottky space of rank $g=1$}\label{Sec:rango1}
Every loxodromic transformation can be conjugated, by a suitable M\"obius transformation, to one (called a normalized one) whose attracting fixed point is $0$ and repelling fixed point is $\infty$. Moreover, any two such normalized loxodromic transformations are conjugated if and only if they are the same. By the normalization, its multiplier $\lambda$ belongs to $\Delta^{*}=\{0<|z|<1\}$, in particular, ${\mathcal M}{\mathcal S}_{1}$ can be identified with $\Delta^{*}$. In this model, the group of holomorphic automorphisms is given by the rotations about the origin, i.e., the transformations of the form $T(z)=e^{i\theta}z$, and the antiholomorphic ones of the form $R(z)=e^{i\theta}\overline{z}$ (all of them real structures whose real parts consist of two components, each one an arc). In this way, there is exactly one real structure, up to conjugation by holomorphic automorphisms, this being represented by the conjugation $J_{1}:\Delta^{*} \to \Delta^{*}$, $J_{1}(\lambda)=\overline{\lambda}$. 

\s
\noindent
\begin{rema}
We have observed that there are exactly six different topologically different extended Schottky groups of rank one. Observe that for $\lambda \in (0,1) \subset \Delta^{*}$, the Schottky group $G_{\lambda}=\langle A_{\lambda}(z)=\lambda z \rangle$ is the orientation preserving half of an extended Schottky groups generated by either (i) two reflections or (ii) two imaginary reflections or (iii) a glide-reflection or (iv) a real Schottky group of type $(+;0;2)$.  Now, if 
for $\lambda \in (-1,0) \subset \Delta^{*}$, then the Schottky group $G_{\lambda}$ is the orientation preserving half of an extended Schottky groups generated by either (i) one reflection and one imaginary reflection or (ii) a real Schottky group of type $(-;1;1)$. 
\end{rema}

\s
%%%%%%%%%%%%%%%%%
\subsection{Marked Schottky space of rank  $g \geq 2$}
If $(A_{1},\ldots,A_{g})$ is a marked Schottky group of rank $g \geq 2$, then we may find a unique M\"obius transformation $M$ so that $MA_{1}M^{-1}$ has its attracting fixed point at $\infty$, $M A_{2} M^{-1}$ has its attracting point at $0$ and $M A_{2}A_{1} M^{-1}$ has its attracting fixed point at $1$. The new marked Schottky group $(MA_{1}M^{-1},\ldots,MA_{g}M^{-1})$, which is equivalent to the previous one, is called a {\it normalized marked Schottky group}.  As any M\"obius transformation is uniquely determined by its action at three different points, in each equivalence class of marked Schottky groups there is exactly one normalized one. In this way, we may consider the map
$$\zeta:\MS \to ({\mathbb C}-\{0,1\})^{3g-3}$$
which sends a normalized (representative) marked Schottky group $(A_{1},\ldots,A_{g})$ to the tuple 
$$(a_{3},\ldots, a_{g}, r_{1},\ldots,r_{g}, s_{2},\ldots,s_{g}),$$ 
where $a_{j}$ and $r_{j}$ are, respectively, the attracting and repelling  fixed points of $A_{j}$ and $s_{j}$ is the repelling fixed points of $A_{j}A_{1}$. This map provides a holomorphic isomorphism of $\MS$ with a domain in $({\mathbb C}-\{0,1\})^{3g-3}$ \cite{H:Schottky} and it permits to see it as a complex manifold of dimension $3(g-1)$.

\s
\noindent
\begin{theo}
If $g \geq 2$, then $\MS$ is a model for the quasiconformal deformation space of Schottky group of rank $g$.
\end{theo}
\begin{proof}
Let us fix a Schottky group $K$ of rank $g \geq 2$ and let $A_{1}$,..., $A_{g}$ be a fixed set of generators of $K$. Up to conjugation by a M\"obius transformation, we may assume $(A_{1},\ldots,A_{g})$ to be normalized.

Given $\mu \in L^{\infty}_{1}(K)$, then we consider the unique quasiconformal homeomorphism $W_{\mu}$, with complex dilation $\mu$, fixing the points $\infty$, $0$ and $1$. The homomorphism $\chi_{\mu}$ is uniquely determined by the values $\chi_{\mu}(A_{1})=W_{\mu} A_{1} W_{\mu}^{-1}, \ldots, \chi_{\mu}(A_{g})=W_{\mu} A_{g} W_{\mu}^{-1}$ and these are uniquely determined by the action of $W_{\mu}$ at the limit set of $K$. As $(\chi_{\mu}(A_{1}),\ldots, \chi_{\mu}(A_{g}))$ is a normalized marked Schottky group of rank $g$, the above provides a (holomorphic) map $L^{\infty}_{1}(K) \to \MS$ which sends $\mu$ to $[(\chi_{\mu}(A_{1}),..., \chi_{\mu}(A_{g}))] \in \MS$. As for every pair of marked Schottky groups, say $(B_{1},...,B_{g})$ and $(C_{1},...,C_{g})$, there is a quasiconformal homeomorphism $W$ with $WB_{j}W^{-1}=C_{j}$, for $j=1,...,g$ (see \cite{Chuckrow}),  we may see that this map is surjective. It is not difficult to see that this map descends to a one-to-one map between ${\mathcal Q}(K)$ and $\MS$.
\end{proof}

\s
%%%%%%%%%%%%%%%%%%%%%
\subsubsection{\bf The group of automorphisms of $\MS$, for $g \geq 2$}
Let $F_{g}=\langle x_{1},\ldots,x_{g}\rangle$ be the free group of rank $g \geq 2$. The group ${\rm Aut}(F_{g})$ of automorphisms of $F_{g}$ is known to be generated by the following automorphisms \cite{MKS}:
$$
\begin{array}{ll}
\rho_{1}: & x_{1} \mapsto x_{2} \mapsto x_{1}, \quad x_{j} \mapsto x_{j}, \; j=2,\ldots,g.\\
\rho_{2}: & x_{1} \mapsto x_{2} \mapsto \cdots \mapsto x_{g-1} \mapsto x_{g} \mapsto x_{1}\\
\rho_{3}: &  x_{1} \mapsto x_{1}^{-1},  \quad x_{j} \mapsto x_{j}, \; j=2,\ldots,g.\\
\rho_{4}: & x_{1} \mapsto x_{1}x_{2},  \quad x_{j} \mapsto x_{j}, \; j=2,\ldots,g.
\end{array}
$$

Earle \cite{Earle} proved, for $g \geq 3$, that  ${\rm Out}(F_{g})={\rm Aut}(F_{g})/{\rm Inn}(F_{g})$, where ${\rm Inn}(F_{g})$ is the subgroup of inner automorphisms of $F_{g}$, is isomorphic to the 
group of holomorphic  automorphisms of $\MS$. 

If $g=2$, then the exterior automorphism $\rho_{0}(x_{1},x_{2})=(x_{1}^{-1},x_{2}^{-1}) \in Out(F_{2})$ induces the identity map of ${\mathcal M}{\mathcal S}_{2}$; in fact, for every marked Schottky group $(A_{1},A_{2})$ it holds that $E=A_{1}A_{2}-A_{2}A_{1} \in {\mathbb M}$ has order two and it satisfies that $EA_{j}E=A_{j}^{-1}$ \cite{Keen}. The group of holomorphic automorphisms of ${\mathcal M}{\mathcal S}_{2}$ is given by ${\rm Out}^{red}(F_{2}):={\rm Out}(F_{2})/\ll \rho_{0} \gg$, where $\ll \rho_{0}\gg$ denotes the normalizer of $\rho_{0}$ in ${\rm Out}(F_{g})$.

\s
\noindent
\begin{rema}\label{obs18}
(1) The abelianizing homomorphism $F_{g} \to {\mathbb Z}^{g}$ provides a natural homomorphism $\theta:{\rm Out}(F_{g}) \to GL_{g}({\mathbb Z})$. Nielsen \cite{Nielsen} proved that $\theta$ is an isomorphism for $g=2$ and surjective for $g \geq 3$. In the last situation, the kernel is a torsion-free \cite{BT} finitely generated group \cite{Magnus}. 

(2) There are finite order elements of ${\rm Out}(F_{g})$ acting freely on $\MS$. For instance, the order two element $\rho(x_{1},x_{3},x_{3})=(x_{1}^{-1},x_{2},x_{3}) \in {\rm Out}(F_{3})$ acts freely. In fact, if there is a fixed point \\ $[(A_{1},A_{2},A_{3})] \in {\mathcal M}{\mathcal S}_{3}$, then there must be a M\"obius transformation $M$ such that $A_{1}^{-1}=MA_{1}M^{-1}$, $A_{2}=MA_{2}M^{-1}$ and $A_{3}=MA_{3}M^{-1}$. Last two equalities ensure that $A_{2}$, $A_{3}$ and $M$ have the same two fixed points; a contradiction as this is not true for $A_{2}$ and $A_{3}$.
\end{rema}

\s
\noindent
\begin{rema}[Quasiconformal deformation space of extended Schottky groups]
As a consequence of the classification done in \cite{H-G:ExtendedSchottky}, there are extended Schottky group of the same rank which are topologically non-equivalent.  This fact makes more difficult to provide a definition of the marked extended Schottky space of a fixed rank. In order to do it, we need to fix a topological type of the extended Schottky group. Once this is done, one may proceed with a definition of marked extended Schottky group and equivalences between them in a similar fashion as done for Schottky groups.  If $K$ is an extended Schottky group of rank $g \geq 2$, then its quasiconformal deformation space ${\mathcal Q}(K)$ is known to be a connected real analytical manifold of real dimension $3(g-1)$. Clearly, if two extended Schottky groups are topologically equivalent, then  there is a real analytic isomorphism between their respective marked extended Schottky spaces. 
\end{rema}

\s
%%%%%%%%%%%%%%%%%%
\subsubsection{\bf Real structures of $\MS$, for $g \geq 2$}\label{obs7}
If $J(z)=\overline{z}$ and $A$ is any M\"obius transformation, then we set $\overline{A}:=J A J$. The bijection
$$J_{g}:\MS \to \MS: [(A_{1},\ldots,A_{g}] \mapsto [(\overline{A_{1}},\ldots,\overline{A_{g}}]$$
induces the {\it canonical real structure}; whose real points are provided by the classes of normalized marked Schottky groups $(A_{1},\ldots,A_{g})$, where $A_{j} \in {\rm PSL}_{2}({\mathbb R})$ (i.e., real Schottky groups whose limit set belongs to the extended real line).

As $J_{g}$ commutes with every holomorphic automorphism of $\MS$, every real structure of $\MS$ has the form
$J_{g} \rho$, where $\rho$ is either the identity (obtaining the canonical real structure) or an order two holomorphic automorphism. Moreover, two real structures 
$J_{g}\rho_{1}$ and $J_{g}\rho_{2}$ are conjugated in the group of holomorphic automorphisms of ${\mathcal M}{\mathcal S}_{g}$ if and only if $\rho_{1}$ and $\rho_{2}$ so are.

\s
\noindent
\begin{rema}[Example]
In Remark \ref{obs18} we have seen that the order two holomorphic automorphism  $\rho(x_{1},x_{2},x_{3})=(x_{1}^{-1},x_{2},x_{3}) \in {\rm Out}(F_{3})$ acts freely on ${\mathcal M}{\mathcal S}_{3}$, but the corresponding real structure $J_{3}\rho$ has real points. In fact, a fixed point of this real structure is provided by a marked Schottky group $(A_{1},A_{2},A_{3})$ so that $\overline{A_{1}}=A_{1}^{-1}$ (i.e., $A_{1}(z)=(az+it)/(isz+\overline{a})$, where $t,s \in {\mathbb R}$) and $A_{2}, A_{3} \in {\rm PSL}_{2}({\mathbb R})$.
\end{rema}

\s

Next,  we apply part (1) of Theorem \ref{teoreal} to describe the real structures of $\MS$ produced by an extended Schottky group of rank $g \geq 2$. 

Let $K$ be an extended Schottky group of rank $g$, $K^{+}$ be its index two orientation-preserving half (a Schottky group of rank $g$) and 
let us assume the extended Schottky group of rank $g$ has signature $(a,b,c,d,e;\gamma_{1},\ldots,\gamma_{e})$  (as in Theorem \ref{maintheo}). Set $A:=a+b$, $B:=c+d$ and $C:=\gamma_{1}+\cdots+\gamma_{e}$.

If we take an element of $K-K^{+}$ and consider its action by conjugation on $K^{+}$, then this provides an order two automorphism $J_{g} \rho_{K} \in {\rm Out}(F_{g})$ (as conjugation by an element of $K^{+}$ induces the identity element of ${\rm Out}(F_{g})$ if we identify $F_{g}$ with $K^{+}$) providing a real structure of $\MS$. If we choose another element in $K-K^{+}$, then we obtain another real structure which is conjugated to the previous  by some element of ${\rm Out}(F_{g})$).  Below, we proceed to describe $\rho_{K}$ in the different topologically conjugacy classes.

Let us assume that in the the construction of $K$ (as in Theorem \ref{maintheo}) we used 
$E_{1},\ldots,E_{A}$ reflections and/or imaginary reflections (if $A>0$), $L_{1},\ldots,L_{c}$ loxodromic transformations (if $c>0$), $N_{1},\ldots, N_{d}$ glide-reflections (if $d>0$), and (if $e>0$) $e$ real Schottky groups $\Gamma_{j}$ generated by a reflection $F_{j}$ and the Schottky group $\Gamma_{j}^{+}=\langle A_{j,1},\ldots,A_{j,\gamma_{j}}\rangle$. 

\begin{enumerate}
\item If $A>0$, then the Schottky group $K^{+}$ is freely generated by the following loxodromic transformations:
$$x_{1}=E_{1}E_{2},\ldots, x_{A-1}=E_{1}E_{A}, x_{A}=L_{1},\ldots,x_{A+c-1}=L_{c}, x_{A+c}=E_{1}N_{1},\ldots, x_{A+B-1}=E_{1}N_{d},$$
$$x_{A+B}=E_{1}L_{1}E_{1},\ldots, x_{A+B+c-1}=E_{1}L_{c}E_{1}, x_{A+B+c}=N_{1}E_{1},\ldots, x_{A+2B-1}=N_{d}E_{1},$$
$$x_{A+2B}=A_{1,1},\ldots, x_{A+2B+\gamma_{1}-1}=A_{1,\gamma_{1}}, x_{A+2B+\gamma_{1}}=A_{2,1},\ldots, x_{A+2B+C-1}=A_{e,\gamma_{e}},$$
$$x_{A+2B+C}=E_{1}F_{1},\ldots, x_{A+2B+C+e-1}=E_{1}F_{e}.$$

In this case, the exterior automorphism $\rho_{K}$, obtained by conjugating by $E_{1}$,  is given by: 
$$\rho_{K}: \left\{ \begin{array}{ll}
x_{j} \mapsto x_{j}^{-1}, & j=1,\ldots, A-1;\\
x_{j} \mapsto x_{j+B+1}, & j=A,\ldots, A+B-1;\\
x_{j} \leftrightarrow x_{j+A+2B+C+1}, & j=A+2B,\ldots, A+2B+C-1;\\
x_{j} \mapsto x_{j}^{-1},& j=A+2B+C,\ldots, A+2B+C+e-1.
\end{array}
\right.
$$

\item If $A=0$ and $e>0$, then the Schottky group $K^{+}$ is freely generated by the following loxodromic transformations:
$$x_{1}=F_{1}F_{2},\ldots, x_{e-1}=F_{1}F_{e}, 
x_{e}=A_{1,1},\ldots,x_{e+\gamma_{1}-1}=A_{1,\gamma_{1}}, 
x_{e+\gamma_{1}}=L_{1},\ldots, x_{e+\gamma_{1}+c-1}=L_{c},$$
$$x_{e+\gamma_{1}+c}=F_{1}N_{1},\ldots, x_{e+\gamma_{1}+B-1}=F_{1}N_{d}, 
x_{e+\gamma_{1}+B}=F_{1}L_{1}F_{1},\ldots, x_{e+\gamma_{1}+B+c-1}=F_{1}L_{c}F_{1},$$
$$x_{e+\gamma_{1}+B+c}=N_{1}F_{1},\ldots, x_{e+\gamma_{1}+2B-1}=N_{d}F_{1},
x_{e+\gamma_{1}+2B}=A_{2,1},\ldots, x_{A+2B+C-1}=A_{e,\gamma_{e}}.$$

In this case, the exterior automorphism $\rho_{K}$, obtained by conjugating by $F_{1}$, is given by: 

$$\rho_{K}: \left\{ \begin{array}{ll}
x_{j} \mapsto x_{j}^{-1}, & j=1,\ldots, e-1;\\
x_{j} \mapsto x_{j}, & j=e,\ldots, e+\gamma_{1}-1;\\
x_{j} \leftrightarrow x_{j+B}, & j=e+\gamma_{1},\ldots, e+B+\gamma_{1}-1;\\
x_{j} \mapsto x_{1}x_{j}x_{1}^{-1},& j=e+2B+\gamma_{1},\ldots, e+2B+\gamma_{1}+\gamma_{2}-1;\\
x_{j} \mapsto x_{2}x_{j}x_{2}^{-1},& j=e+2B+\gamma_{1}+\gamma_{2},\ldots, e+2B+\gamma_{1}+\gamma_{2}+\gamma_{3}-1;\\
\vdots & \vdots\\
x_{j} \mapsto x_{e-1}x_{j}x_{e-1}^{-1},& j=e+2B+\gamma_{1}+\cdots\gamma_{e-1},\ldots, e+2B+C-1.
\end{array}
\right.
$$

\item If $A=e=0$, then (remember that in this case, we may assume $c=0$; so $B=d$)
the Schottky group $K^{+}$ is freely generated by the following loxodromic transformations:
$$x_{1}=N_{1}^{2}, x_{2}=N_{1}N_{2},\ldots, x_{B}=N_{1}N_{B}, x_{B+1}=N_{1}N_{2}^{-1},\ldots, x_{2B-1}=N_{1}N_{B}^{-1}.$$

In this case, the exterior automorphism $\rho_{K}$, obtained by conjugating by $N_{1}$, is given by:

$$\rho_{K}: \left\{ \begin{array}{ll}
x_{1} \mapsto x_{1};\\
x_{j} \mapsto x_{1}x_{B+j-1}^{-1}, & j=2,\ldots, B;\\
x_{j} \mapsto x_{1}x_{j-B+1}^{-1}, & j=B+1,\ldots, 2B-1.
\end{array}
\right.
$$

\end{enumerate}

\s
\noindent
\begin{rema}\label{obs:conjugar}
It can be seen from the above explicit description that two extended Schottky groups of the same rank and signatures
$(a_{1},b_{1},c_{1},d_{1},e_{1};\gamma_{1,1},\ldots,\gamma_{1,e_{1}})$ and $(a_{2},b_{2},c_{2},d_{2},e_{2};\gamma_{2,1},\ldots,\gamma_{2,e_{2}})$
induce the same real structure on $\MS$ if
$$a_{1}+b_{1}=a_{2}+b_{2}, \quad c_{1}+d_{1}=c_{2}+d_{2}, \quad e_{1}=e_{2}, \quad \{\gamma_{1,1},\ldots,\gamma_{1,e_{1}}\}=\{\gamma_{2,1},\ldots,\gamma_{2,e_{2}} \}.$$

If $g=2$, then the other direction does not hold (see Section \ref{obs21}). For $g \geq 3$ we believe that the reciprocal holds, that is, there would be exactly 
$U_{g}=\left[ \frac{g+3}{2}\right]+\sum_{f \in \Delta_{g}} \left[ \frac{g_{f}+3}{2}\right]$ real structures with non-empty real part  (up to conjugation), where $g_{f}$ and $\Delta_{g}$ are the same as defined in Lemma \ref{lemita10}.
 \end{rema}

\s
%%%%%%%%%%%%%%%%%%%
\subsection{Real structures of ${\mathcal M}{\mathcal S}_{2}$: proof of part (3) of Theorem \ref{realpoints}}\label{obs21}
As for $g=2$ the group of holomorphic and antiholomorphic automorphisms is isomorphic to  ${\rm Out}^{red}(F_{2}) \times {\mathbb Z}_{2}$, every real structure has the form $J_{2}\rho$, where $\rho \in {\rm Out}^{red}(F_{2})$ is either the identity or has order two. 

The classification of finite order elements of ${\rm Out}(F_{2})$, up to conjugation, was obtained in \cite{Meskin}. It says that, up to conjugation, there exactly three elements of order two, one of order three, one of order four and one of order six. The elements of order two in ${\rm Out}(F_{2})$ are represented by 
$\rho_{0}(x_{1},x_{2})=(x_{1}^{-1},x_{2}^{-1})$, $\rho_{1}(x_{1},x_{2})=(x_{2},x_{1})$ and $\rho_{2}(x_{1},x_{2})=(x_{1}^{-1},x_{2})$ and that of order four is represented by $\rho_{3}(x_{1},x_{2})=(x_{2}^{-1},x_{1})$. The elelemnt $\rho_{0}$ induces the identity on ${\rm Out}^{red}(F_{2})$ and, since $\rho_{3}^{2}=\rho_{0}$, the elements $\rho_{1}$, $\rho_{2}$ and $\rho_{3}$ induce the following three involutions in ${\rm Out}^{red}(F_{2})$
$$\widetilde{\rho_{1}}([A_{1},A_{2})])=[(A_{2},A_{1})], \quad
\widetilde{\rho_{2}}([A_{1},A_{2})])=[(A_{1}^{-1},A_{2})], \quad
\widetilde{\rho_{3}}([A_{1},A_{2})])=[(A_{2}^{-1},A_{1})]=[(A_{2},A_{1}^{-1})].$$

As $\widetilde{\rho_{3}}=\widetilde{\rho_{1}} \widetilde{\rho_{2}}$, these last three involutions generate a copy of ${\mathbb Z}_{2}^{2}$ inside ${\rm Out}^{red}(F_{2})$.
By composing these three ones and the identity with the canonical real structure, we obtain the following four real structures on ${\mathcal M}{\mathcal S}_{2}$:
$$J_{2}([(A_{1},A_{2})])=[(\overline{A_{1}},\overline{A_{2}})], \quad 
\widehat{\rho_{1}}([(A_{1},A_{2})])=[(\overline{A_{2}},\overline{A_{1}})],$$
$$\widehat{\rho_{2}}([(A_{1},A_{2})])=[(\overline{A_{1}}\;^{-1},\overline{A_{2}})], \quad 
\widehat{\rho_{3}}([(A_{1},A_{2})])=[(\overline{A_{2}}\;^{-1},\overline{A_{1}})]$$

\s
\noindent
\begin{lemm}
The real structure  $\widehat{\rho_{3}}$ has no real points and 
the real structures $J_{2}$, $\widehat{\rho_{1}}$ and $\widehat{\rho_{2}}$
have non-empty real part. Moreover, the real structure $J_{2}$ has exactly five connected components of real points, $\widehat{\rho_{1}}$ has three and $\widehat{\rho_{2}}$ has two.

\end{lemm}
\begin{proof}
Assume by the contrary that $\widehat{\rho_{3}}$ has a fixed point of . It is provided by a marked Schottky group $(A_{1},A_{2})$ so that there is a M\"obius transformation $M$ such that $MA_{1}M^{-1}=\overline{A_{2}}\;^{-1}$ and $MA_{2}M^{-1}=\overline{A_{1}}$. We may assume the attracting fixed point of $A_{1}$ and $A_{2}$ are $\infty$ and $1$, respectively, and that the repelling fixed points are, respectively, $0$ and $r \in {\mathbb C}-\{0,1\}$. It follows that $M:(\infty,0,1,r) \mapsto (\overline{r},1,\infty,0)$. This asserts that 
$|r|=1$ and $M(z)=(z-r)/(rz-r)$. The equalities $MA_{1}M^{-1}=\overline{A_{2}}\;^{-1}$ and  $MA_{2}M^{-1}=\overline{A_{1}}$ now  asserts that $A_{1}(z)=z/(r-1)$ and $A_{2}(z)=(-r^{2}z+r(2r-1))/((1-2r)z+r^{2}+r-1)$. As $\infty$ is the attracting fixed points of $A_{1}$, it follows that $|r-1|<1$. In any case, the fact that $|r|=1$ will obligates to have that the real part (if non-empty) of $\widehat{\rho_{3}}$ should have real dimension one; a contradiction to the general fact that its dimension should be three.

We already noted that $J_{2}$ has fixed points. In order to see that $\widehat{\rho_{1}}$ and $\widehat{\rho_{2}}$
have non-empty real part, we use the structural description of extended Schottky groups of rank two of Theorem \ref{maintheo}, together the computations done in Section \ref{obs7}. In fact, the extended Schottky groups of rank two corresponding to the signatures $(a,b,c,d,e;\gamma_{1},\ldots,\gamma_{e})$ as in Theorem \ref{maintheo} induce the real structures as follows:
$$(3,0,0,0,0;-), (2,1,0,0,0;-), (1,2,0,0,0;-), (0,3,0,0,0;-), (0,0,0,0,1;2) \mapsto J_{2}$$
$$(1,0,1,0,0;-), (1,0,0,1,0;-), (0,1,0,1,0;-) \mapsto \widehat{\rho_{1}}$$
$$(1,0,0,0,1;1), (0,1,0,0,1;1) \mapsto \widehat{\rho_{2}}$$

The above asserts that the real structure $J_{2}$ has exactly five connected components of real points, that $\widehat{\rho_{1}}$ has three and $\widehat{\rho_{2}}$ has two.
\end{proof}

\s
%%%%%%%%%%%%%%%%%%%
\subsection{Real structures of $\MS$, for $g\geq 3$: First half of part (4) of Theorem \ref{realpoints}}\label{obs22}
As for $g \geq 3$ the group of holomorphic and antiholomorphic automorphisms of $\MS$ is isomorphic to ${\rm Out}(F_{g}) \times {\mathbb Z}_{2}$, the number of conjugacy classes of real structures of $\MS$ is equal to $1+T_{g}$, where $T_{g}$ is the number of conjugacy classes of elements of order two of ${\rm Out}(F_{g})$.
This provides the  first half of part (4) of Theorem \ref{realpoints}.

%%%%%%%%%%%%%%%%%%%%%%%%%%
%%%%%%%%%%%%%%%%%%%%%%%%%%
\section{Proof of Theorem  \ref{realpoints}}\label{sec:prueba}

Let us fix a Schottky group $K$ of rank $g \geq 2$, with region of discontinuity $\Omega$. Let $S=\Omega/K$ be the closed Riemann surface uniformized by $K$ and  let $T:\Omega \to S$
be a regular covering with $K$ as deck group. 
In this case, $M=M_{K}$ is a handlebody of genus $g$ and its conformal boundary is the closed Riemann surface $S$.

As a consequence of Proposition \ref{identifica}, we may identify ${\mathcal T}(M)$ with ${\mathcal Q}(K)$ (which can also be identified with $\MS$). In this case, we can make this more precise. There is a natural homomorphism $\theta:{\rm Diff}(M) \to {\rm Diff}(S)$, defined by the restriction; $\theta(f)=f|_{S}$. If $f \in {\rm Diff}^{+}(M)$, then $\theta(f) \in {\rm Diff}^{+}(S)$. The image $\theta({\rm Diff}(M))$ is exactly the subgroup of ${\rm Diff}(S)$ consisting of those diffeomorphisms which extends continuously as diffeomorphisms of $M$. As $M$ is a compression body, it is not difficult to observe that if $\theta(f) \in {\rm Diff}_{0}(S)$, then $f \in {\rm Diff}_{0}(M)$.  It follows that there is a natural one-to-one homomorphism $\theta:{\rm Mod}(M) \to {\rm Mod}(S)$.
Let $H(M)$ be the (necessarily normal) subgroup of ${\rm Mod}(M)$ consisting of those classes of diffeomorphisms of $M$ inducing the identity homomorphism on $\pi_{1}(M) \cong K$. It follows that $H(M) \cong \theta(H(M))=H(K) \lhd {\rm Mod}(S)$, where $H(K)=H(K,\Omega)$ was defined in a previous section. As noted before, we have a universal covering map
$\Pi:{\mathcal T}(S) \to {\mathcal Q}(K)$,  with $\theta(H(M)^{+})=\theta(H(M))^{+}=H(K)^{+}$ as the deck group, that is, $H(M)^{+} \cong \pi_{1}(\MS)$ \cite{Maskit:Guanajuato}. 

\s
\noindent
\begin{rema}
Every real structure of $\MS$ (with non-empty real part) lifts, by the universal cover  $\Pi:{\mathcal T}(S) \to \MS$, to a real structure of ${\mathcal T}(S)$. As the cover group $H(K)^{+}$ is not a normal subgroup of the modular group, there are real structures of ${\mathcal T}(S)$ which do not induce real structures of $\MS$. Also, it might be that two conjugate real structures of ${\mathcal T}_{g}$ induce non-conjugate real structures on $\MS$.  The main reason of this is that there are diffeomorphisms of $S$ which cannot extend continuously to the handlebody $M$.  
\end{rema}

\s
%%%%%%%%%%%%%%
\subsection{Proof of part (1) and second half of part (4)}
Let us consider a real structure $\tau^{*}$ on ${\mathcal Q}(K)$ with a real point $p=[\mu] \in{\mathcal Q}(K)$  (i.e., $\tau^{*}(p)=p$) and let $q \in {\mathcal T}(S)$ be so that $\Pi(q)=p$. Consider a lifting of $\tau^{*}$, say  
$\widetilde{\tau}:{\mathcal T}(S) \to {\mathcal T}(S)$, which fixes the point $q$; in particular, $\widetilde{\tau}$ is a real structure on the Teichm\"uller space ${\mathcal T}(S)$. As a consequence of Royden's theorem \cite{Royden}, the locus of real points of such a real structure is a real analytic submanifold  of ${\mathcal T}(S)$. By Nielsen's isomorphism theorem \cite{Kerckhoff}, there exists an orientation reversing  homeomorphism $\widehat{\tau}:S \to S$ of order two inducing the real structure $\widetilde{\tau}$. The fixed point $q$ of $\widetilde{\tau}$ corresponds to a (maybe different) Riemann surface structure on $S$ for which $\widehat{\tau}$ acts as anti-conformal involution. The new Riemann surface structure on $S$ induces a Schottky structure on $M$, say provided by the Schottky group $K_{p}=W_{\mu} K W_{\mu}^{-1}$, where $W_{\mu}$ is a normalized quasiconformal homeomorphism with Beltrami coefficient $\mu$. The involution $\widehat{\tau}$ extends continuously to an orientation-reversing isometry (given by $K_{p}$) of order two on the interior of $M$ (whose isotopy class is given by $\tau^{*}$).

In this way, the anti-conformal involution $\widehat{\tau}$ (on the corresponding Riemann surface structure) can be lifted to an anti-conformal automorphism $\eta_{p}$ of $\Omega_{p}=W_{\mu}(\Omega)$ which normalizes the Schottky group $K_{p}$. As the region of discontinuity of a Schottky group is of class $O_{AD}$ (that is, it admits no holomorphic function with finite Dirichlet norm \cite[pg 241]{A-S:RS}), 
every conformal map from $\Omega_{p}$ into the extended complex plane is a M\"obius transformation \cite[pg 200]{A-S:RS}, in particular, 
$\eta_{p}$ is an extended M\"obius transformation.  If $\widehat{K}_{p}=\langle K_{p},\eta_{p}\rangle$, then $\widehat{K}_{p}$ is an extended Kleinian group whose region of discontinuity is $\Omega_{p}$ so that $K_{p}$ is its 
orientation-preserving half. This ensures that each real point of the real structure $\tau^{*}$ can be identified with an extended Schottky  group whose orientation-preserving half is a quasiconformal deformation of $K$ and that $\tau^{*}$ is induced by such an extended Schottky group.

\s
%%%%%%%%%%%%%%%
\subsection{Proof of part (2)}
Let $\tau_{1}^{*}$ and $\tau_{2}^{*}$ be two real structures of $ {\mathcal Q}(K)$ and 
let $[\mu_{j}] \in {\mathcal Q}(K)$ be a fix point of $\tau_{j}^{*}$, for $j=1,2$. Assume that the corresponding extended Schottky groups $\widehat{K}_{[\mu_{1}]}$ and $\widehat{K}_{[\mu_{2}]}$ (as constructed above) are topologically conjugated by an orientation-preserving homeomorphism $f:W_{\mu_{1}}(\Omega) \to W_{\mu_{2}}(\Omega)$. 
The homeomorphism $W_{\mu_{2}}^{-1} f W_{\mu_{1}}:\Omega \to \Omega$ descends to an orientation-preserving homeomorphism  $h:S \to S$  conjugating the corresponding orientation reversing involutions $\widehat{\tau}_{j}$ induced by these two real structures. It follows that the class $[h] \in H^{*}(K)^{+}$ induces, by Proposition \ref{prop1}, a holomorphic automorphism of ${\mathcal Q}(K)$ conjugating the two given real structures.

%%%%%%%%%%%%%%%%%
\subsection{A remark with respect to part (4)}
Let us assume that $g \geq 3$. Above we have seen that every real structure with non-empty real part is 
 induced by an extended Schottky group. By  
Royden's theorem, the group of holomorphic and anti-holomorphic automorphisms of $\MS$ can be identified with ${\rm Diff}(M)/H(M)$ and that there is a bijection between the conjugacy classes of real structures of $\MS$, with non-empty real part, and the conjugacy classes of orientation reversing diffeomorphisms of order two of $M$ up to isotopy. The following provides another argument to obtain Part (4).

\s
\noindent
\begin{lemm}\label{lemita}
Let $f:M \to M$ be an orientation reversing diffeomorphism of order two. Then there is an extended Schottky group $\Gamma$ and a diffeomorphism 
$h:M \to M_{\Gamma^{+}}$ so that $hfh^{-1}$ is induced by $\Gamma$. In particular, every real structure of $\MS$ with non-empty real part is induced by some extended Schottky group.
\end{lemm}
\begin{proof}
 Let $f_{0}$ be the restriction of $f$ to the boundary $\partial(M)$.  
By Nielsen's realization Theorem \cite{Kerckhoff},
there is a Riemann surface structure $S$ on $\partial(M)$ and a
anti-conformal involution $\eta:S \to S$, isotopic to $f_{0}$. It is well known that there is a Schottky group $G$ providing an uniformization of $S$ for which $\eta$ lifts. The lifting produces an extended Schottky group $\Gamma$ with $G=\Gamma^{+}$. The Schottky group $G$ provides a Schottky structure on $M$ whose conformal boundary is $S$ and for which $\eta$ extends to $M$ as an anti-conformal involution. As $\eta$ is homotopic to $f_{0}$ on the boundary and $M$ is a compression body,  the result follows.
\end{proof}

\s
\noindent
\begin{rema}\label{observacion}
In the proof of Theorem \ref{realpoints}, we have used the fact that every holomorphic (respectively, antiholomorphic) automorphism of the region of discontinuity of a Schottky group is the restriction of a M\"obius (respectively, extended M\"obius) transformation. So, the arguments can be used for function groups with such a property; for instance Kleinian groups which are the free product of a Schottky group and some finite copies of Kleinian groups isomorphic to ${\mathbb Z}^{2}$ (the uniformized Riemann surfaces still closed ones and the Kleinian manifold is a handlebody from which some simple loops have been deleted). This will be pursued elsewhere.
\end{rema}

\s
 
%%%%%%%%%%%%%%%%%%%%%%%%%%%%%%%%%%%%
%%%%%%%%%%%%%%%%%%%%%%%%%%%%%%%%%%%

\end{document}